\documentclass[11pt]{amsart}
\usepackage[utf8]{inputenc}
\usepackage[top=1in, bottom=1in, left=1in, right=1in]{geometry}

\usepackage{amsmath,amsfonts,amssymb,graphicx,csquotes,mathtools}
\usepackage{mathrsfs}
\usepackage{enumerate}

\usepackage{fixltx2e}

\usepackage[dvipsnames]{xcolor} 
\usepackage{hyperref} 

\usepackage{color}
\usepackage[dvipsnames]{xcolor}
\usepackage{tkz-graph}
\usepackage{tikz-cd}
\usetikzlibrary{decorations.pathreplacing,calligraphy}

\usepackage{graphicx}
\graphicspath{plots}
\usepackage{subcaption}

\usepackage[
backend=bibtex,
sorting=nyt
]{biblatex}
\makeatletter
\def\namedlabel#1#2{\begingroup
    #2%
    \def\@currentlabel{#2}%
    \phantomsection\label{#1}\endgroup
}
\makeatother

\usepackage[colorinlistoftodos]{todonotes}

\usepackage{amsthm}

\theoremstyle{definition}
\newtheorem{defn}{Definition}[section]

\theoremstyle{plain}
\newtheorem{lem}[defn]{Lemma}
\newtheorem{sublem}[defn]{Sublemma}

\newtheorem*{thm*}{Theorem}
\newtheorem{prop}[defn]{Proposition}
\theoremstyle{plain}
\newtheorem{cor}[defn]{Corollary}

\theoremstyle{remark}
\newtheorem{rem}[defn]{Remark}
\theoremstyle{definition}
\newtheorem{exm}[defn]{Example}

\makeatother

\numberwithin{defn}{section}

\newtheorem{manualtheoreminner}{Theorem}
\newenvironment{manualtheorem}[1]{%
  \renewcommand\themanualtheoreminner{#1}%
  \manualtheoreminner
}{\endmanualtheoreminner}

\newcommand{\quand}{\quad\text{ and } \quad}

\numberwithin{equation}{section}

\title[Doubly intermittent maps with regularly varying tails]{Doubly Intermittent Maps with Critical
Points, Unbounded Derivatives and Regularly Varying Tail}

\begin{document}

\author{Mubarak Muhammad}
\address{Abdus Salam International Centre for Theoretical Physics (ICTP), Trieste, Italy and  Scuola Internazionale Superiorie di Studi Avanzati (SISSA), Trieste, Italy}
\email{mubarak@aims.edu.gh}

\author{Tanja I.\ Schindler}
\address{University of Vienna, Austria}
\email{tanja.schindler@univie.ac.at}

\begin{abstract}
We consider a class of interval maps with up to two indifferent fixed points, unbounded derivative and regularly varying tails. Under some mild assumptions we prove the existence of a unique mixing absolutely continuous invariant measure and give conditions under which the measure is finite. Moreover, in the finite measure case we give a formula for the measure-theoretical entropy and upper bounds for a very slow decay of correlations. 
This extends former work by Coates, Luzzatto and Muhammad to maps with regularly varying tails. Particularly, we investigate the boundary case where the behaviour of the slowly varying function decides on the finiteness of the measure and on the decay of correlation. 
\end{abstract}

\thanks{The second author was supported by the Austrian Science Fund FWF: P 33943-N.
Furthermore, this research was partly done during (mutual) visits of the authors to the Abdus Salam International Center for Theoretical Physics, Trieste, the University of Pisa (support through the “visiting fellows” program for T.S. and support for the minicourse participation "Mutually enhancing connections between Ergodic Theory, Combinatorics and Number Theory" for M.M.) and the University of Vienna (support from Austrian Science Fund FWF: P 33943-N for M.M.) and during the conference DinAminicI VII at the Riemann International School of Mathematics, Varese. 
}

\maketitle

\tableofcontents

\section{Introduction and Statement of Results}
Interval maps with one intermittent fixed point (e.g.\ Pomeau-Manville maps) are a very classical object of study in finite and infinite ergodic theory, see e.g.\ \cite{LivSauVaiProbabilisticApproachIntermittency1999, YouRecurrenceTimesRates1999} for statistical and mixing results for such maps. In their standard representation they admit a $\sigma$-finite absolutely continuous invariant measure (acim) which - depending on the precise behaviour at the indifferent fixed point - can be finite or infinite.
Because of the vastness of literature in this area we don't give a complete literature survey about this topic. 

Closely related to that, also interval maps with two (or more) indifferent fixed points became an object of interest early on, see e.g.\ \cite{thaler_transformations_1983}. 
Under the condition that both fixed points show enough intermittency 
those maps can also be seen as paradigmatic examples for maps with an infinite acim with the additional property of having a dynamically separating set, see e.g.\ \cite{aaronson_occupation_2005} for a definition. Those maps show additional statistical properties compared to interval maps with only one indifferent fixed point, see e.g.\ \cite{thaler_distributional_2006, sera_functional_2020, aaronson_occupation_2005, BonGiulLen}.

On the other hand, a further generalisation has been studied in the physics literature modeling around others anomalous diffusion and geofluid dynamics  \cite{physics0, physics1, physics2, pikovsky}: 
Additionally to the intermittent fixed point(s) 
it is allowed there that the derivative is unbounded at one point. These two properties interplay with each other in the way that even if there is an intermittent fixed point with a very small derivative it is possible that the acim is still finite if the derivative has a pole of a large enough order, see \cite{pikovsky, sandropaper, coaluzmub22}. \newline

In this paper we extent the definition of the maps introduced in \cite{coaluzmub22} (maps with two intermittent fixeds point and with unbounded derivatives) to a more robust class. In particular, we introduce a slowly varying function near the intermittent fixed point. To study such maps we need additional techniques, e.g.\ Karamata Theory, see e.g.\ \cite{bingham1989regular}.

 Particularly, we investigate the boundary case where the behaviour of the slowly varying function decides if the invariant measure is finite or not. 
 Moreover, for such maps we give a very slow upper bound for the decay of correlations. 
 We use similar methods as in \cite{holland2005slowly} where an easier class of functions (maps with only one indifferent fixed point and a bounded derivative) is studied. 
For such maps we prove the existence of a finite acim and mixing properties for a boundary parameter of the intermittent maps. 
In general, those maps have a very slow decay of correlations. 
 
We will start by introducing the precise maps we aim to study and then state our main theorems.

\subsection{Maps with Regularly Varying Tails}\label{subsec: reg var tails}
\begin{figure}[ht]
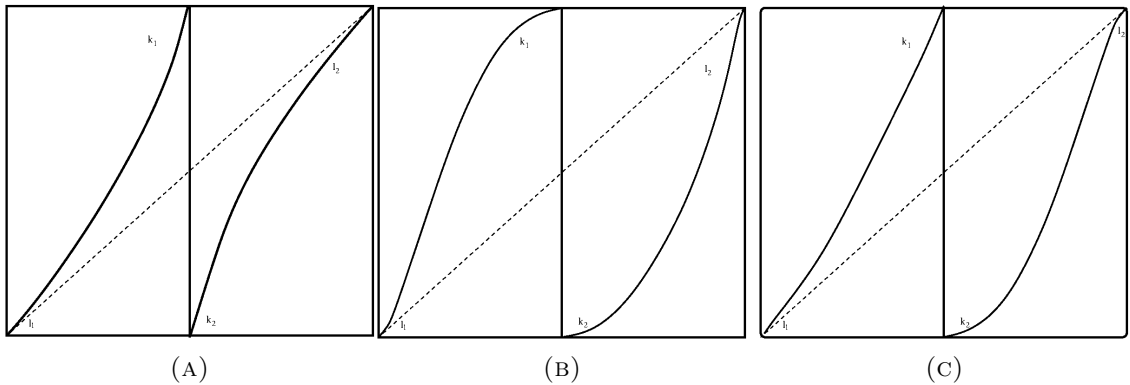

\centering
  \begin{subfigure}{.3\textwidth}
    \centering
    \includegraphics[width=.99\linewidth]{a.pdf}
    \caption{}\label{a}
  \end{subfigure}%
  \begin{subfigure}{.3\textwidth}
    \centering
    \includegraphics[width=.99\linewidth]{b.pdf}
    \caption{}\label{b}
  \end{subfigure}
  \begin{subfigure}{.3\textwidth}
    \centering
    \includegraphics[width=.99\linewidth]{c.pdf}
    \caption{}\label{c}
  \end{subfigure}
  \caption{Graph of $g$ for various possible values of parameters.}\label{fig:fig}
\end{figure}

First, we say that a function $\mathcal{L}$ is \emph{slowly varying} at $0$ (at $\infty$ respectively) if for all $d>0$ we have $\lim_{x\to 0} \mathcal{L}(dx)/\mathcal{L}(x)=1$ (we have $\lim_{x\to \infty} \mathcal{L}(dx)/\mathcal{L}(x)=1$ respectively). We say that a function $\mathcal{R}$ is \emph{regularly varying} at $0$ with index $\gamma$ if for all $d>0$ we have $\lim_{x\to 0} \mathcal{R}(dx)/\mathcal{R}(x)=d^{\gamma}$ and analogously for regular variation at $\infty$.

Let \(I=[-1, 1]\), \(I_{-}=[-1, 0]\), \(I_{+}=[0,1].\) Then their interiors \( \mathring I, \mathring I_-, \mathring I_+\) fulfill \(\mathring I_{-}\cap\mathring I_{+}=\emptyset\) and we have \(I=I_{-}\cup I_{+}\).
We could easily generalize the results to other compact intervals with these properties, however, for brevity we stick to this definition.

\noindent Furthermore, we assume there exists $g:I\to I$ satisfying the following properties:

\begin{description}
    \item [\namedlabel{itm:A0'}{\textbf{(A0)}}] \( g: I \to I \) is \emph{full branch} such that \(g_{-}=g_{\big|\mathring I_-}: \mathring I_{-}\to \mathring I\) and \(g_{+}=g_{\big|{\mathring I_+}}: \mathring I_{+}\to \mathring I\) are  orientation preserving \( C^{2} \) diffeomorphisms with only $\{-1,1\}$ as fixed points. 
    \item [\namedlabel{itm:A1'}{\textbf{(A1)}}]
    There exist constants \(\ell_1,\ell_2, \iota, k_1,k_2  > 0 \) and functions \[\mathcal{L}_{-1}:U_{-1}{+1}\to [0,\infty) \quand \mathcal{L}_{1}:U_{+1}{-1}\to [0,\infty)\] 
    both slowly varying at $0$
    such that
        \begin{enumerate}[(i)]
      \item
        \begin{equation}\label{eqn_1___}
          g(x) =
          \begin{cases} 
            x+{(1+x)^{1+\ell_1}}\mathcal{L}_{-1}(1+x) &  \text{in }  U_{-1},\\
            1-a_1{|x|}^{k_1} & \text{in } U_{0-}, \\
            -1+a_2{x}^{k_2} & \text{in }  U_{0+}, \\
            x-{(1-x)^{1+\ell_2}}\mathcal{L}_{1}(1-x) &  \text{in }  U_{+1},
          \end{cases}
        \end{equation}
          where
          \begin{equation}\label{eq:U'}
            U_{0-}:=(-\iota, 0],
            \quad
            U_{0+}:=[0, \iota), 
            \quad
            U_{-1}:=g(U_{0+}), 
            \quad
            U_{+1}:=g(U_{0-}).
          \end{equation}
        \item 
          For some $M>0$ the functions \(\mathcal{L}_{-1}\left(1/{\cdot}\right)\) and \(\mathcal{L}_{1}\left(1/{\cdot}\right)\) are chosen to be $C^{\omega}$ (real analytic) on $[M,\infty)$.

          For further ease we restrict $U_{0+}$ and $U_{-1}$ (and $U_{0-}$ and $U_{+1}$) to such an area such that there exits $\epsilon>0$ fulfilling
\begin{equation}\label{eq: slow var cond}
    g'(-1+a_2x^{k_2})-1\geq \epsilon (a_2 x^{k_2})^{\ell_1}\mathcal{L}_{-1}(a_2x^{k_2})\qquad \text{ (and } \qquad 
    g'(1-a_1|x|^{k_1})-1\geq \epsilon (a_1 x^{k_1})^{l_2}\mathcal{L}_{+1}(a_1x^{k_1}) \text{)},
\end{equation}
for all $x\in U_{-1}$ (and $x\in U_{+1}$).  
          \end{enumerate}
\end{description}

\begin{rem}
Observe that the slow variation property of  \(\mathcal{L}_{-1}, \mathcal{L}_{1}\) comes into play at $\pm 1$ because $1\mp x\to 0$ as $x\to \pm 1$.
The last condition is mainly a condition on the slowly varying functions. As the slow variation is an asymptotic property for $x\to -1$ and $x\to 1$, the representation in \eqref{eqn_1___} says nothing about the behaviour of $g$ further away from $-1$ and $1$. 
However, using Lemmas \ref{1.5.8} and \ref{1.5.10} and Remark \ref{Rel_slow@0_and_infity}, we have 
$$g'(x)-1\sim   (\ell_1+1) (1 + x)^{\ell_1} \mathcal{L}_{-1}(1 + x),\qquad\text{ for }x\to-1.$$
Here, we denote by $a(x)\sim b(x)$ that $\lim_{x\to -1}a(x)/b(x)=1$. In the following, we will also use the $\sim$-notation for $x$ tending to other values than $-1$ which should however be clear from the context.  
Thus, $g'(-1+a_2x^{k_2})-1\sim   (a_2 x^{k_2})^{\ell_1}\mathcal{L}_{-1}(a_2x^{k_2})$ and \eqref{eq: slow var cond} has to hold on some subregions $\widetilde{U}_{-1}\subset U_{-1}, \widetilde{U}_{1}\subset U_{1}$. Shrinking then $U_{-1}$ and $U_{-1}$ to $\widetilde{U}_{-1}$ and $\widetilde{U}_1$ (and $U_{0}, U_{0-}$ accordingly) gives a concise representation allowing \eqref{eqn_1___} and \eqref{eq: slow var cond} to hold as the same time. 

On the other hand, we note that if $\mathcal{L}_{-1}'\geq 0$, then we have on $U_{-1}$ that 
\begin{align*}
 g'(x)&=1+(1+\ell_1)(1+x)^{\ell_1}\mathcal{L}_{-1}(1+x)+ (1+x)^{\ell_1+1}\mathcal{L}_{-1}'(1+x)\\
 &\geq  1+(1+\ell_1)(1+x)^{\ell_1}\mathcal{L}_{-1}(1+x)
\end{align*}
and \eqref{eq: slow var cond} immediately holds without any further restrictions on $U_{-1}$. This is in particular the case when $\mathcal{L}_{-1}$ is a positive constant. An analogous argumentation for $\mathcal{L}_{1}$ and $U_1$ holds.

Furthermore, notice that when \(\mathcal{L}_{-1}, \mathcal{L}_{1}\) are positive constants, \eqref{eqn_1___} corresponds to maps in \cite{coaluzmub22} with \(\mathcal{L}_{-1}=b_1\) and \(\mathcal{L}_{1}=b_2\). 
\end{rem}

The last assumption we will impose is a generalisation of saying that \( g \) is uniformly expanding outside the neighbourhoods \( U_{0\pm}\) and \( U_{\pm 1}\) which is however more than we need. 
The assumption is basically equivalent to \cite[Assumption (A2)]{coaluzmub22}, however, as we need the notation in future, we will give it in full detail describing some of the topological structure of maps satisfying Condition \ref{itm:A0'}. We start by defining 
\begin{equation}\label{eq:Delta0_}
\Delta^-_0:= 
g^{-1}(0,1)\cap I_-
\quand 
\Delta^+_0:= 
\end{equation}
and define then iteratively, for every \( n \geq 1 \),  the sets 
\begin{equation}\label{eq:Delta_}
 \Delta_n^{-}:= g^{-1}(\Delta_{n-1}^{-})\cap I_{-}
 \quand 
 \Delta_n^{+}:= g^{-1}(\Delta_{n-1}^{+})\cap I_{+}
 \end{equation}
 as the \( n\)'th preimages of \( \Delta_0^-, \Delta_0^+\) inside the intervals \(I_{-}, I_{+} \).   It follows from \ref{itm:A0'} that 
\( 
 \{ \Delta_n^{-}\}_{n\geq 0}
\) 
and \( 
 \{ \Delta_n^{+}\}_{n\geq 0}
\) 
are $\bmod\;0$ partitions of \(I_{-}\) and \(I_{+}\) respectively, and that the partition elements depend \emph{monotonically} on the index  in the sense that \( n > m \) implies that \( \Delta_n^{\pm}\) is closer to \( \pm 1\) than \( \Delta_m^{\pm}\),  in particular, the only accumulation points of these partitions are \( -1\) and \( 1 \) respectively.  
Then, for every \( n \geq 1 \),  we let 
\begin{equation}\label{eq:delta_}
 \delta_{n}^{-}:= 
 g^{-1}(\Delta_{n-1}^{+}) \cap  \Delta_0^{-} 
 \quand 
 \delta_{n}^{+}:= 
 g^{-1}(\Delta_{n-1}^{-}) \cap  \Delta_0^{+}.
\end{equation}
Observe  that  
 \(
 \{ \delta_n^{-}\}_{n\geq 1}  
\) 
and 
\( 
 \{ \delta_n^{+}\}_{n\geq 1} 
\)
are $\bmod\; 0$ partitions of \( \Delta_0^-\) and \( \Delta_0^+\)  respectively and also in these cases  the partition elements depend monotonically on the index in the sense that \( n > m \) implies that \( \delta_n^{\pm}\) is closer to \( 0 \) than \( \delta_m^{\pm}\)   (and in particular the only accumulation point of these partitions is 0). Notice moreover, that 
\[
 g^{n}(\delta_{n}^{-})= \Delta_{0}^{+} \quad\text{and} \quad   g^{n}(\delta_{n}^{+})= \Delta_{0}^{-}.  \]  

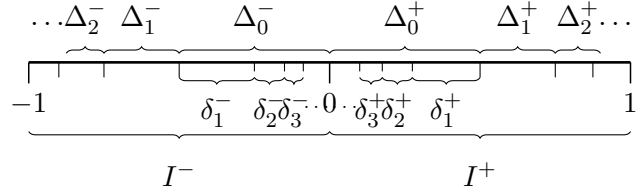
\begin{figure}[h]
 \begin{tikzpicture}[x=1cm,y=0.4cm]
 \draw[line width=0.35mm] (-4,0)-- (4,0);
 \draw[line width=0.2mm] (-4,0) -- (-4,-0.7) node[below] {$-1$};
 \draw[line width=0.2mm] (0,0) -- (0,-0.7) node[below] {$0$};
 \draw[line width=0.2mm] (4,0) -- (4,-0.7) node[below] {$1$};

 \draw (2,0) -- (2,-0.5);
 \draw (3,0) -- (3,-0.5);
  \draw (3.5,0) -- (3.5,-0.5); 

 \draw (-2,0) -- (-2,-0.5); 
 \draw (-3,0) -- (-3,-0.5);
 \draw (-3.6,0) -- (-3.6,-0.5);

  \draw (-1,0) -- (-1,-0.3);
 \draw (-0.6,0) -- (-0.6,-0.3);
 \draw (-0.35,0) -- (-0.35,-0.3); 

  \draw [decorate,     decoration = {brace}] (-1,-0.5) --  (-2,-0.5);
 \draw (-1.5,-0.7) node[anchor=north] {$\delta_1^-$};
 \draw [decorate,     decoration = {brace}] (-0.6,-0.5) --  (-1,-0.5);
 \draw (-0.8,-0.7) node[anchor=north] {$\delta_2^-$}; 
  \draw [decorate,     decoration = {brace}] (-0.35,-0.5) --  (-0.6,-0.5);
 \draw (-0.475,-0.7) node[anchor=north] {$\delta_3^-$}; 
 \draw (-0.175,-1.1) node[anchor=north] {\tiny $\ldots$};

  \draw (1.1,0) -- (1.1,-0.3);
 \draw (0.7,0) -- (0.7,-0.3);
 \draw (0.4,0) -- (0.4,-0.3); 
 
 \draw [decorate,     decoration = {brace}] (2,-0.5) --  (1.1,-0.5);
 \draw (1.55,-0.7) node[anchor=north] {$\delta_1^+$};
 \draw [decorate,     decoration = {brace}] (1.1,-0.5) --  (0.7,-0.5);
 \draw (0.9,-0.7) node[anchor=north] {$\delta_2^+$}; 
  \draw [decorate,     decoration = {brace}] (0.7,-0.5) --  (0.4,-0.5);
 \draw (0.55,-0.7) node[anchor=north] {$\delta_3^+$}; 
 \draw (0.2,-1.1) node[anchor=north] {\tiny $\ldots$};

 \draw [decorate,     decoration = {brace}] (0,0.3) --  (2,0.3);
  \draw (1,0.6) node[anchor=south] {$\Delta_0^+$};
  \draw [decorate,     decoration = {brace}] (2,0.3) --  (3,0.3);
  \draw (2.5,0.6) node[anchor=south] {$\Delta_1^+$};
  \draw [decorate,     decoration = {brace}] (3,0.3) --  (3.6,0.3);
  \draw (3.3,0.6) node[anchor=south] {$\Delta_2^+$};
  \draw (3.8, 0.9) node[anchor=south] {$\ldots$};

  \draw [decorate,     decoration = {brace}] (-2,0.3) --  (0,0.3);
  \draw (-1,0.6) node[anchor=south] {$\Delta_0^-$};
  \draw [decorate,     decoration = {brace}] (-3,0.3) --  (-2,0.3);
  \draw (-2.5,0.6) node[anchor=south] {$\Delta_1^-$};
  \draw [decorate,     decoration = {brace}] (-3.5,0.3) --  (-3,0.3);
  \draw (-3.25,0.6) node[anchor=south] {$\Delta_2^-$};
  \draw (-3.75, 0.9) node[anchor=south] {$\ldots$};

  \draw [decorate,     decoration = {brace}] (0,-2.3) --  (-4,-2.3);
  \draw (-2,-3) node[anchor=north] {$I^-$};
    \draw [decorate,     decoration = {brace}] (4,-2.3) --  (0,-2.3);
  \draw (2,-3) node[anchor=north] {$I^+$};

  \end{tikzpicture}
 \caption{A schematic sketch of the sets on the interval}
\end{figure}

We now define two non-negative integers \( n_{\pm}\) which depend on the positions of the partition elements \( \delta_{n}^{\pm}\) and on the sizes of the neighbourhoods \( U_{0\pm}\) on which the map \( g \) is explicitly defined.  
If  \( \Delta_0^{-} \subseteq U_{0-}\) and/or  \( \Delta_0^{+} \subseteq U_{0+}\), we define   \( n_{-}= 0  \) and/or \( n_{+}=0\) respectively, otherwise we let 
\begin{equation}\label{eq:n+-_}
 n_{+} := \min \{n :\delta_{n}^{+} \subset U_{0+} \} \quand 
  n_{-} := \min \{n :\delta_{n}^{-} \subset U_{0-} \}.  
 \end{equation}
We can now formulate our final assumption as follows. 
  \begin{description}
    \item[\namedlabel{itm:A2'}{\textbf{(A2)}}]  
There exists a \( \lambda >  1 \) such that  for all   \( 1\leq n\leq n_{\pm}\)   and for all \(    x \in \delta_n^{\pm} \)  we have \(  (g^n)'(x) > \lambda\). 
 \end{description}
 \begin{rem}
   This condition is a generalization of saying that $g$ has to be expanding outside the neighbourhoods  \( U_{0\pm}\) and \( U_{\pm 1}\). However, we may assume that $k_1>1$ but $a_1$ is so small that $g'(\pm \iota)<1$ and by continuity there exists $\epsilon>0$ such that $g'(\pm \iota \pm \epsilon)<1$. 
However, Condition \ref{itm:A2'} is a condition also on points close to $\pm 1$. As $(g^n)'(x)= g'(g^{n-1}(x))\cdot g'(x)$ having a large expansion close to $U_{\mp 1}$ might compensate the small derivative near $U_{0\pm}$.  
 \end{rem}

Furthermore, let \[\widehat{\mathfrak G}:=\{g:I\to I \text{ satisfying } \ref{itm:A0'}-\ref{itm:A2'}\}.\]
We remark that \(\widehat{\mathfrak F}\cap\{\ell_1,\ell_2>0\}\subset \widehat{\mathfrak G}\) where \(\widehat{\mathfrak F}\) is the class of maps introduced in \cite{coaluzmub22}.

\subsection{Statement of Results}
Our first  result is completely general and applies to all maps in \( \widehat{\mathfrak{G}} \).

\begin{manualtheorem}{A}\label{thm:main1'}
Every \( g \in \widehat{\mathfrak G} \) admits a unique (up to scaling by a constant) invariant \( \sigma \)-finite measure which is equivalent to the Lebesgue measure $m$.
\end{manualtheorem}
Notice that this result extends \cite[Theorem A]{coaluzmub22} to a broader class than $g\in \widehat{\mathfrak F}\cap\{\ell_1,\ell_2>0\}$. By our construction we have that the density with respect to Lebesgue of the measure given in Theorem \ref{thm:main1'} is locally Lipschitz and unbounded only at the endpoints \( \pm 1\). Depending on the slowly varying functions \(\mathcal{L}_{-1}\) and \(\mathcal{L}_{1}\), we will see that the density may or may not be integrable and thus the measure may or may not be finite. More specifically, let 
\[
  \beta_1 \coloneqq k_2 \ell_1,\quad
  \beta_2 \coloneqq k_1 \ell_2,
  \quand  \beta \coloneqq \max\{ \beta_{1}, \beta_{2} \}.
\]

\noindent We note that for every function $\mathcal{L}$ slowly varying in infinity there is another function $\mathcal{L}^\#$ slowly varying in infinity with properties given in Lemma \ref{1.5.13} and equivalently for functions slowly varying in zero. We call \(\left(\mathcal{L}, \mathcal{L}^{\#}\right)\) a \emph{de Bruijn conjugate pair}. 
We also note that if $\mathcal{L}$ is a constant (and in a number of other cases), then $\mathcal{L}^{\#}$ can be chosen as $\mathcal{L}^{\#}=1/\mathcal{L}$.

\noindent 
For the following we define 
\begin{equation}\label{eq: def L pm 1}
    L_1(n):= \mathcal{L}_{1}^{\#}(n^{-\frac{1}{\ell_2}})^{\frac{1}{k_1}}\quad\text{ and }\quad  L_{-1}(n):=\mathcal{L}_{-1}^{\#}({ n}^{-\frac{1}{\ell_1}})^{\frac{1}{k_2}}
\end{equation}
and note that $L_1$ and $L_{-1}$ are both slowly varying in infinity. 
Furthermore, let 
\begin{equation}\label{chi_1_2}
    \chi^{1}_{n}\coloneqq n^{-\frac1{\beta_2}}L_1(n), \qquad \chi^{2}_{n}\coloneqq n^{-\frac1{\beta_1}}L_{-1}(n)
\end{equation}
and 
\[
 { { \mathfrak G}} := \bigg\{g\in \widehat {\mathfrak G}: \exists \bar{n}_0,n_0\in \mathbb{N}: \max\bigg\{\sum_{n=\bar{n}_0}^{\infty} \chi^{1}_{n}, \sum_{n=n_0}^{\infty} \chi^{2}_{n}\bigg\}<\infty\bigg\}.
\]

\noindent Observe that \( \mathfrak G\) contains all \(g\in\widehat{\mathfrak G}\) with \(\beta<1\). In addition, it contains also functions $g$ where $\beta_1=1$ or $\beta_2=1$ (or possibly both). In this case the growth rate of
$\mathcal{L}_{1}^{\#}$ (or $\mathcal{L}_{-1}^{\#}$) decides if $\sum_{n=1}^{\infty} \chi^{1}_{n}<\infty$ (or $\sum_{n=1}^{\infty} \chi^{2}_{n}<\infty$ respectively) holds or not. 
Notice that when \(\mathcal{L}_{-1}\) is a positive constant, we have \(\chi^{2}_{n}\asymp n^{-\frac1{\beta_1}}\) and the condition is the same as in \cite{coaluzmub22}. Here and in the following we write $a(x)\asymp b(x)$ if there exists a constant $C>1$ such that $C^{-1} a(x)\leq b(x)\leq C a(x)$ for all $x$ in the respective domain.

\begin{manualtheorem}{B}\label{thm:main2'}
A map  \( g \in \widehat{\mathfrak G}\)  admits a unique ergodic invariant \emph{probability} measure \( \mu_g \) equivalent to the Lebesgue measure $m$ \emph{if and only if} \( g \in \mathfrak G\). 
\end{manualtheorem}

\noindent Notice that of particular interest is the case when $\beta=1$ which is a boundary case in \cite{coaluzmub22} where acip measure cease to exist. However, introducing the slowly varying functions $\mathcal{L}_{-1}, \mathcal{L}_1$ creates a spectrum of new parameters within $\beta=1$ with a more subtle boundary base on the slowly varying functions. Observe also that the condition \( \beta=1\) is a restriction only on the \emph{relative} values of \( k_{1}\) with respect to \( \ell_{2}\) and of \(k_{2}\) with respect to \( \ell_{1}\).  It still allows  \( k_{1}\) and/or \( k_{2} \) to be \emph{arbitrarily large}, therefore allowing more ``degenerate'' critical points, as long as the  corresponding exponents \( \ell_{2} \) and/or \( \ell_{1}\) are significantly small, i.e.\ provided the corresponding neutral fixed points are not so degenerate.
Furthermore, we notice that the condition on $\mathcal{L}_{-1}, \mathcal{L}_1$ changes according to the precise choice of $k_1,k_2,\ell_1,\ell_2$ if $k_1\ell_2=1$ or $k_2\ell_1=1$.

\medskip

\noindent
For the following we will consider the case that the invariant measure is finite and we will give two results about the statistical properties of $g\in \mathfrak{G}$ with respect to its invariant measure $\mu_g$. 

First we consider the \emph{measure-theoretic entropy} of \( g \) with respect to a measure \( \mu \) which is defined as 
\[
h_{\mu}(g):= \lim_{n\to\infty}\sup_{\mathcal P} \left\{ - \frac{1}{n} \log \sum_{\omega_{n}\in \mathcal P_{n}} \mu(\omega_{n})\log \mu(\omega_{n})\right\}.
\]
Here the supremum is taken over all finite measurable partitions \( \mathcal P\) of the underlying measure space and \( \mathcal P_{n} := \mathcal P \vee g^{-1}\mathcal P\vee \cdots \vee g^{-n}\mathcal P \) is the dynamical refinement of \( \mathcal P\) by \( g \). With that we can formulate the following theorem:

\begin{manualtheorem}{C}\label{thm:mainentropy}
  Let \( g \in \mathfrak G\). Then \( \mu_g \) (defined as in Theorem \ref{thm:main2'})  satisfies the  \emph{Pesin entropy formula:}
  \(
      h_{\mu_g}(g)=\int \log |g'|\mathrm{d}\mu_g.
  \)
\end{manualtheorem}

\medskip
\noindent
 For the following let $\mathcal{H}$ be the set of H\"older continuous  functions from $I$ to the reals.
\noindent For \( \varphi_1, \varphi_2\in\mathcal{H} \)   and \( n \geq 1 \), we define the \emph{correlation function}
        \[
       \mathcal C_{n}(\varphi_1, \varphi_2)
       := \left| \int \varphi_1\circ g^{n} \varphi_2\,\mathrm{d}\mu_g - \int \varphi_1 \,\mathrm{d}\mu_g \int \varphi_2 \,\mathrm{d}\mu_g \right|.
        \]
Furthermore, define
\begin{equation}\label{gen_SVF}
L(n)\coloneqq
\begin{cases}
\max\left\{L_{-1}(n),L_1(n)\right\}&\text{if }\beta_1=\beta_2,\\
 L_{-1}(n) &\text{if }\beta_1>\beta_2,\\
     L_1(n) & \text{if }\beta_1<\beta_2
\end{cases}    
\end{equation}
and introduce the notation $a(x)\lesssim b(x)$ if
there exists a constant $C>0$ such that $a(x)\leq Cb(x)$ for all $x$ in the respective domain.

With this we obtain the following theorem concerning the decay of correlation:
\begin{manualtheorem}{D}\label{thm:main2''}
Let \(\varphi_1,\varphi_2\in\mathcal{H}\), and \( g \in \mathfrak G\). Then we have 
    \begin{equation}\label{eq: Cn gen}
    \mathcal C_{n}(\varphi_1, \varphi_2)\lesssim 
        \begin{cases}
         \sum_{l\ge n}l^{-1}L(l), \quad \text{if} \quad \beta= 1,\\[5pt]
            n^{-\frac{1}{\beta}+1}L(n), \quad\quad \text{if} \quad \beta\neq 1.
        \end{cases}
    \end{equation}
We note that by Lemma \ref{1.5.9b} $\sum_{l\ge n}l^{-1}L(l)$ is a slowly varying function at $\infty$. Furthermore, if $\mathcal{L}_{\pm1}$ are both constant (i.e. $\beta\neq 1$) then \eqref{eq: Cn gen} simplifies to 
    \[ \mathcal C_{n}(\varphi_1, \varphi_2)\lesssim n^{-\frac{1}{\beta}+1}\,.
    \]
\end{manualtheorem}

On the other hand, if $\beta=1$ we obtain a very slow upper bound for the decay of correlation as the following examples show. There, we denote by $\log^k(x)$ the $k$th iterate of the logarithm. 

\begin{exm}\label{ex: Example1}
 Let $g\in \mathfrak{G}$ be such that $\beta=1$, let
 \begin{align}\label{eq: def kappa}
  \kappa\coloneqq \begin{cases}
                   \min\left\{k_1,k_2\right\}&\text{if }\beta_1=\beta_2,\\
  k_2 &\text{if }\beta_1>\beta_2,\\
      k_1 & \text{if }\beta_1<\beta_2
                  \end{cases}
 \end{align}
 and let us assume that 
\begin{align*}
\mathcal{L}_{-1}(x)=\mathcal{L}_1(x)=  (\log^r (1/x))^{\kappa+\alpha} \prod_{j=1}^{r-1}(\log^j (1/x))^{\kappa}.
\end{align*}
If $\alpha=0$, then the invariant measure $\mu_g$ is infinite. 
If $\alpha>0$, the invariant measure $\mu_g$ is finite and for any $\varphi_1,\varphi_2\in\mathcal{H}$ we have 
\begin{align*}
 \mathcal{C}_n(\varphi_1,\varphi_2)\lesssim (\log^r(n))^{-\alpha/\kappa}. 
\end{align*}
\end{exm}

\begin{exm}\label{ex: Example2}
Let $g\in \mathfrak{G}$ be such that $\beta=1$, let $\kappa$ be as in \eqref{eq: def kappa} and let us assume for $r\geq 2$ and $\alpha\in (0,1)$ that
 \begin{align}
  \mathcal{L}_{-1}(x)
  =\mathcal{L}_1(x)
  =\left(\exp\left((\log^r (1/x)^{\alpha}\right)\, (\log^r (1/x))^{1-\alpha}\prod_{j=1}^{r-1}\log^j(1/x)\right)^{\kappa}.\label{eq: def lambda}
 \end{align}
Then the invariant measure $\mu_g$ is finite and for any $\varphi_1,\varphi_2\in\mathcal{H}$ we have 
\begin{align}
 \mathcal{C}_n(\varphi_1,\varphi_2)\lesssim \exp(-(\log^r n)^{\alpha}). \label{eq: dec exm2}
\end{align}
\end{exm}
The two examples above are analogous to the examples which are given in \cite[Theorem 2]{holland2005slowly}. It can be seen however that the results are not completely analogous and that the values of $k_1,k_2$ play a role on how the upper bound of the decay of correlations can be estimated and if the measure is finite or not. 

Finally, we give an example where even the precise values of $\ell_1,\ell_2$ has an influence on the behaviour of the systems. (The reason that $\ell_1,\ell_2$ were not appearing in the previous two examples is that in those cases $\mathcal{L}_{\pm 1}(x)\sim \mathcal{L}_{\pm 1}(x^{\ell})$, for any $\ell>0$.)

\begin{exm}\label{ex: Example3}
Let $g\in \mathfrak{G}$ be such that $\beta=1$, let $\kappa$ be as in \eqref{eq: def kappa} and let
\begin{align*}
 \lambda\coloneqq \begin{cases}
                   \max\left\{\ell_1,\ell_2\right\}&\text{if }\beta_1=\beta_2,\\
  \ell_1 &\text{if }\beta_1>\beta_2,\\
      \ell_2 & \text{if }\beta_1<\beta_2.
                  \end{cases}
\end{align*}
Additionally, let us assume that $\alpha\in (0,1/2)$ and
 \begin{align*}
  \mathcal{L}_{-1}(x)
  =\mathcal{L}_1(x)
  =\exp\left(\kappa\left(\lambda^{\alpha} (\log (1/x))^{\alpha}+(1-\alpha)\log\log (1/x)\right)\right).
 \end{align*}
Then we obtain 
 \begin{align*}
  \mathcal{C}_n(\varphi_1,\varphi_2)\lesssim \exp(-(\log n)^{\alpha}).
 \end{align*}
It would be possible to also construct examples where we obtain with the same techniques an upper bound of $\exp(-(\log n)^{\alpha})$ for $\alpha\in [1/2,1)$. We will remark on that at the end of Section \ref{sec: proof of thm D}. However, as the calculations become lengthy and it is our main goal to give an idea on how the parameters $k_1,k_2,\ell_1,\ell_2$ come into play for the case $\beta=1$, we leave the precise calculation to the interested reader. 
\end{exm}

\begin{rem}\label{rem: lower bound dec corr}
In the above examples we were giving only upper bounds for the decay of correlation.
In case of $\beta<1$, there are techniques which give a polynomial lower bound, see e.g.\ \cite{gouezel_polestim_2004, sarig_subexdec_2002, HuVainti_lowdec_2019} which were also used in \cite{sandropaper}.
However, to the authors' knowledge, there are no results yet which give a slowly varying lower bound for a decay of correlations and as stated in \cite{gouezel_polestim_2004} the techniques there are not sufficient to prove that the slowly varying decays of correlations obtained in \cite{holland2005slowly} are optimal and hence can neither be used in our case. So further work in this direction is needed.
\end{rem}

\begin{rem}
 The previous results show that the maps behave quite regularly. However, we assume that in a number of cases they might behave differently to e.g.\ well behaved piecewise expanding interval maps. 
For instance, as discussed later in Section \ref{sec:exp__}, Adler's condition as in \cite{adler} cannot hold and one thus expects a different behaviour of these maps as the one described e.g.\ in \cite{thaler_transformations_1983}.
Moreover, also dynamical Borel-Cantelli lemmas for nested sets have been investigated for Pikovsky maps, see \cite{BC-lemmas} showing a non-standard behaviour. Hence, we expect that in particular results where the existence of consecutive maxima and strong mixing properties play a central role, e.g.\ in the trimmed sum context \cite{trimming_paper, BonSchi}, will substantially change compared to piecewise expanding interval maps or, in the infinite measure case, compared to maps with an indifferent fixed point but a bounded derivative.
\end{rem}

\subsection{Structure of the paper}
The paper is structured as follows: In Section \ref{sec:outline} we give the main steps for the proof of Theorems \ref{thm:main1'}, \ref{thm:main2'} and \ref{thm:mainentropy}. For doing so we give two main propositions, Proposition \ref{thm:inducedmap'} and \ref{tail_proposition_} about the structure of the induced map and the tails of the return times. 
The proofs of those propositions are given in Sections \ref{sec:induced} and \ref{sec:tail__est} respectively. 
Since the proof of Theorem \ref{thm:main2''} requires a number of technical estimates from Sections \ref{sec:induced} and \ref{sec:tail__est} we postpone it to the last Section \ref{sec: proof of thm D} in which we also give the proof of Examples \ref{ex: Example1} to \ref{ex: Example3}.

 \section{Overview of the proof}
\label{sec:outline}
\noindent Here we explain our overall strategy and give some key technical propositions with which we prove our theorems. The proof of the propositions is however postponed to a later part of the paper.

First, 
 we construct the first return induced map \( G \) to $\Delta_0$, show that it is a full branch map with infinitely many branches, prove asymptotic estimates related to the construction and finally show that \( G \) is uniformly expanding and has bounded distortion.

\noindent Define the first return time 
\begin{equation}\label{eq: tau}
 \tau: \Delta_0^-\to \mathbb{N} \quad \text{by} \quad
\tau (x) \coloneqq \min \{ n > 0 : g^{n} (x) \in \Delta_0^{-}\}   
\end{equation} 
and the \emph{first return induced map} 
\begin{equation}\label{eq:G__}
G: \Delta_0^- \to \Delta_0^- \quad \text{ by } \quad G(x) \coloneqq g^{\tau(x)} (x).
\end{equation}
An induced map is said to saturate the interval $I$ if 
\begin{equation}\label{eq:sat__}
 \bigcup_{n\geq 0} \bigcup_{i= 0}^{n-1} g^{i}(\{\tau =  n \})
= \bigcup_{n\geq 0} g^{n}(\{\tau >   n \})
   = I \ (\text{mod} \ 0). 
  \end{equation}
  Intuitively, saturation means that the return map ``reaches'' every part of the original domain of the map \( g \), and thus the properties and characteristics of the return map reflect, to some extent, all the relevant characteristics of \( g \). 
  
  \begin{rem}\label{rem:disjoint'}
   If \( G \) is a first return induced map, as in our case, then any two sets of the form \( g^{i}(\{\tau =  n \}) \), $i=0,\ldots,n-1$ either coincide or are disjoint and therefore those sets form a partition of \( I \) mod 0.
  \end{rem}

\noindent The following proposition will be a main result leading to the proof of Theorem \ref{thm:main1'} and Theorem \ref{thm:main2'}.

\begin{prop}\label{thm:inducedmap'}  Let \( g \in \widehat{\mathfrak{G}} \). Then  
  \( G: \Delta_0^- \to \Delta_0^-  \)  given as in \eqref{eq:G__} is a first return induced  Gibbs-Markov map which saturates \( I \). 
\end{prop}

We give the precise definition of  Gibbs-Markov maps and prove Proposition \ref{thm:inducedmap'} in Section \ref{sec:induced}. 
In Section \ref{sec:top__} we describe the topological structure of \( G \) and show that it is a full branch map with countably many branches which saturates \( I \) (we will define \( G \) as a composition of two full branch maps, see \eqref{def:Gtilde__} and \eqref{def:G-}, which is why we call the construction  a double inducing procedure); in Section \ref{sec:est__} we obtain key estimates concerning the sizes  of the partition elements of the corresponding partition; in Section \ref{sec:exp__} we show that \( G \) is uniformly expanding; in Section~\ref{sec:dist'__} we show that \(G \)  has bounded distortion. From these results we get Proposition~\ref{thm:inducedmap'} from which we can then obtain our first main Theorem \ref{thm:main1'}.

\begin{proof}[Proof of Theorem \ref{thm:main1'}]
Let \( g \in \widehat{\mathfrak{G}} \). By Proposition \ref{thm:inducedmap'}, \(G\) is a Gibbs-Markov map which saturates \( I \). Together with Theorem 3.13 in \cite{alves2020nonuniformly} we have that \( G \) admits a unique ergodic invariant probability measure \( \hat{\mu}_- \), supported on \( \Delta_0^- \),  which is equivalent to the Lebesgue measure~\( m \) and which has   Lipschitz continuous density 
\( 
    \hat h_-= \mathrm{d}\hat\mu_-/\mathrm{d}m
\)   bounded above and below. 

\noindent We then  ``spread'' the measure over the original interval \( I \) by defining the  measure 
\begin{equation}
  \label{eq:mu'}
  \tilde\mu  \coloneqq \sum_{n=0}^{\infty} 
  g^n_*(\hat\mu_-|\{\tau \geq  n\})
\end{equation}
where \( 
 g^n_*(\hat\mu_-|\{\tau \geq  n\})(E):=   \hat{\mu}_- ( g^{-n} ( E ) \cap \{ \tau \geq  n \} ). 
\) By Theorem 3.18 in \cite{alves2020nonuniformly},  we have that \( \tilde\mu \) is a $\sigma$-finite measure which is ergodic and $g$-invariant and absolutely continuous with respect to Lebesgue. The fact that  \( G\) saturates \( I \) implies moreover that  \( \tilde \mu \) is equivalent to Lebesgue, which completes the proof. 
\end{proof}

\begin{rem}\label{rem:inducedmap'}
We can define the first return map 
\(
G_+: \Delta_0^+ \to \Delta_0^+ \) in a completely analogous way to the definition of \( G\) above. Moreover, the conclusions of Proposition \ref{thm:inducedmap'} hold for \( G_+\) and thus \( G_+\) admits a unique ergodic invariant probability measure \( \hat{\mu}_+ \) which is equivalent to Lebesgue measure~\( m \) and such that the density \(\hat h_+:= \mathrm{d}\hat\mu_+/\mathrm{d}m\)
is Lipschitz continuous and bounded above and below. 
The two maps \( G\) and \( G_+\) are clearly distinct, as are the measures \( \hat \mu_-\) and \( \hat \mu_+ \), but exhibit a  subtle kind of symmetry in the sense that the corresponding  measure \( \tilde \mu\) obtained by substituting \( \hat \mu_- \) by \( \hat \mu_+ \) in \eqref{eq:mu'} is,  up to  a constant scaling factor, exactly the same measure. This originates from the fact that $\tilde{\mu}$ is $g$-invariant and absolutely continuous with respect to Lebesgue and thus uniquely determined.
\end{rem}

\begin{cor}\label{cor:density}
Let \( g \in \widehat{\mathfrak{G}} \). The density \( \tilde h \) of \( \tilde \mu|_{\Delta_0^- \cup \Delta_0^+} \) is Lipschitz continuous and bounded. Moreover, \( \tilde \mu |_{\Delta_0^-} = \hat \mu_-\). 
\end{cor}

\begin{proof}
The proof follows in exactly the same manner as Corollary 2.4 in \cite{coaluzmub22}.
\end{proof}

\begin{rem}\label{rem:mu'_}
We have used above the notation \( G \) rather than \( G_{-}\) for simplicity as this is the map which plays a more central role in our construction. Similarly, we will from now on simply use the notation \( \hat\mu\) to denote the measure 
\( \hat \mu_- \). 
\end{rem}

A further important proposition is the following:
\begin{prop}\label{tail_proposition_} Let \( g \in \widehat{\mathfrak{G}} \). Then there are constants $C_1,C_2>0$ such that
          \[\tilde \mu(\tau>n)\sim C_1 \chi^{1}_{n} + C_2\chi^{2}_{n}\]
with \((\chi^1_n),(\chi^2_n)\) defined in \eqref{chi_1_2}.
\end{prop}
\noindent We prove Proposition \ref{tail_proposition_} in Section \ref{sec:tail__est}. We now use it to prove Theorem \ref{thm:main2'}. 
\begin{proof}[Proof of Theorem \ref{thm:main2'} and \ref{thm:mainentropy}]
Let \( g \in \widehat{\mathfrak{G}} \). As \( g^{-n}(I)=I \), we have by the definition of $\tilde\mu$ in \eqref{eq:mu'} that
\[
   \tilde\mu(I)    = \sum_{n=0}^{\infty} \hat\mu_-(g^{-n}(I)\cap\{\tau > n\})
   =\sum_{n=0}^{\infty} \hat\mu_-(I\cap\{\tau > n\})
                    =\sum_{n=0}^{\infty} \hat\mu_-(\tau > n).
   \]
  Together with Proposition \ref{tail_proposition_} we have  
  \(
   \tilde\mu(I)\sim C_1\sum_{n=0}^{\infty} \chi^{1}_{n}+C_2\sum_{n=0}^{\infty} \chi^{2}_{n}
   \)
   which is bounded if and only if $g\in \mathfrak{G}$. We now define the required measure by \( \mu_{g}:=\tilde \mu/\tilde \mu(I)\) which gives the statement of Theorem \ref{thm:main2'}.

   To prove Theorem \ref{thm:mainentropy} we aim to apply Theorem A in \cite{AlvMes21} which states that in our case finiteness of $h_{\mu_g}(g)$ is already enough for Pesin's entropy formula to hold.
   \(\mathcal{D}\coloneqq \{(-1,0),(0,1)\}\) is a Lebesgue mod $0$  generating partition with \(H_{\mu_g}(\mathcal{D})<\infty\) 
   where for some partition $\mathcal{E}$, we have $H_{\mu_g}(\mathcal{E}):=-\sum_{E\in\mathcal{E}}\mu(E)\log\mu(E)$.
   Using then Lemma 1.19 and Proposition 2.4 of \cite{bowen} gives $h_{\mu_g}(g)<\infty$ and thus Theorem \ref{thm:mainentropy}.
\end{proof}

As we need for the proof of Theorem \ref{thm:main2''} some further statements which will only be proven later, we postpone this proof to Section \ref{sec: proof of thm D}

\section{The Induced Map and proof of Proposition~\ref{thm:inducedmap'}}\label{sec:induced}

In this section we prove Proposition~\ref{thm:inducedmap'}.  
We begin by recalling one of several equivalent definitions of Gibbs-Markov maps. 
\medskip

\noindent First, if we assume that $\mathscr{P}$ is a partition of $I$, we define the \emph{separation time} $s: I^2\to \mathbb{N}_0$ to be
\begin{equation}\label{sep-time}
   s ( x, y ) \coloneqq
  \inf \{ n \geq 0 : G^n x \text{ and } G^n y \text{ lie in different elements of the partition } \mathscr{P} \}. 
\end{equation}
With this we are able to give the definition of Gibbs-Markov maps.
  
\begin{defn}\label{def_GM}
An interval map \( F: I \to I \) is called a (full branch) Gibbs-Markov map if there exists a partition \( \mathcal P \) of \( I \) (mod 0) into open subintervals such that: 
\begin{enumerate}
\item \( F \) is \emph{full branch}: for all \(\omega \in \mathcal P \) the restriction \(F|_{\omega}: \omega \to \mathring I \) is a \( C^{1}\) diffeomorphism; 
\item \( F \) is \emph{uniformly expanding}: there exists \( \lambda > 1 \) such that \( |F'(x)|\geq \lambda\) for all \( x\in \omega \) for all \( \omega \in \mathcal P\);
\item \( F \) has \emph{bounded distortion}: 
there exist \( C>0, \theta \in (0,1) \) such that for all $\omega\in \mathcal P$  and  all $x, y\in\omega $, 
\[\log \bigg|\frac{F'(x)}{F'(y)}\bigg|\leq  C\theta^{s(x,y)}.
\]
\end{enumerate}
\end{defn}

We will show that the first return map \( G\) defined in \eqref{eq:G__} satisfies all the conditions above as well as the saturation condition \eqref{eq:sat__}. 
In  Section \ref{sec:top__} we describe the topological structure of \( G \) and show that it is  a full branch map with countably many branches which saturates \( I \); this will require only the very basic topological structure of \( g \) provided by Condition~\ref{itm:A0'}. In Section~\ref{sec:est__} we obtain estimates concerning the sizes  of the partition elements of the corresponding partition; this will require the explicit form of the map \( g \) as given in \ref{itm:A1'}. In Section \ref{sec:exp__} we show that \( G \) is uniformly expanding; this will require the final condition \ref{itm:A2'}. Finally, in Section~\ref{sec:dist'__} we use the estimates and results obtained  to show that \(G \)  has bounded distortion.

\subsection{Topological Construction}\label{sec:top__}
In this section we make some topological construction which depends only on condition \textbf{(A0)}. The construction is the same as in \cite{coaluzmub22}. However, since we need the notation for our following calculation, we give it in full detail.

\noindent Remember the definitions of $\Delta_n^{\pm}$, $n\geq 0$ from \eqref{eq:Delta0_} and \eqref{eq:Delta_} and $\delta_n^{\pm}$ from \eqref{eq:delta_}.

Note that for every  \( n \geq 1 \),  
\[
 g:\delta_{n}^{-} \to  \Delta_{n-1}^{+},
\quad 
 g:\delta_{n}^{+} \to  \Delta_{n-1}^{-},
\quad
g^{n-1}: \Delta_{n-1}^{-}\to \Delta_0^{-}, 
\quand
g^{n-1}: \Delta_{n-1}^{+}\to \Delta_0^{+}
\]
are \( C^{2} \) diffeomorphisms. Thus, \( 
g^{n}: \delta_n^{-} \to \Delta_0^+ 
\) and 
\( 
g^{n}: \delta_n^{+} \to \Delta_0^- 
\) are also  \( C^{2} \) diffeomorphisms. These give rise to two \emph{full branch induced maps}
\begin{equation}\label{def:Gtilde__}
\widetilde G^-:\Delta_{0}^{-} \to \Delta_{0}^{+} \quand 
\widetilde G^+:\Delta_{0}^{+} \to \Delta_{0}^{-}
\end{equation}
defined by 
\(
\widetilde G^-|_{\delta_{n}^{-} }= g^{n} 
\) and \( 
\widetilde G^+|_{\delta_{n}^{+} }= g^{n} 
\). 
For all \(m, n \geq 1\) we define 
\[
\delta^-_{m,n} := (G^-)^{-1}(\delta^+_n) \cap \delta^-_m 
\quand
\delta^+_{m,n} := (G^+)^{-1}(\delta^-_n) \cap \delta^+_m.
\] 
Then 
 \(
 \{ \delta_{m,n}^{-}\}_{n\geq 1}  
\) 
and 
\( 
 \{ \delta_{m,n}^{+}\}_{n\geq 1} 
\)
are partitions for \( \delta_m^-\) and \( \delta_m^+\)  respectively and thus 
\[
\mathscr P^- :=  \{ \delta_{m,n}^{-}\}_{m,n\geq 1} 
 \quand 
 \mathscr P^+ :=  \{ \delta_{m,n}^{+}\}_{m,n\geq 1} 
\]
are partitions of \( \Delta_0^-, \Delta_0^+\)  respectively, such that for every \( m,n \geq 1\), the maps 
\[
g^{m+n}: \delta_{m,n}^{-} \to \Delta_{0}^{-}
\quand 
g^{m+n}: \delta_{m,n}^{+} \to \Delta_{0}^{+}
\]
are \( C^2\) diffeomorphisms. These further give rise to two \emph{full branch induced maps} 
\begin{equation}\label{def:G-}
G^-:=\widetilde G^+ \circ \widetilde G^- :\Delta_{0}^{-} \to \Delta_{0}^{-} \quand G^+:=\widetilde G^- \circ \widetilde G^+ :\Delta_{0}^{+} \to \Delta_{0}^{+}
\end{equation}
with return time $R:\Delta^{\pm}\to \mathbb{Z}$ defined as 
\[
R|_{\delta_{m,n}^{-} }= {m+n} 
\quand
R|_{\delta_{m,n}^{+} }= {m+n}.  
\]

\subsection{Partition Estimates}\label{sec:est__}
In this section we give asymptotic estimates of the partition which depend on the form of the map given in \ref{itm:A1'}. Let 
\( ( x_n^{-} )_{n \geq 0} \text{ and }  ( x_{n}^+ )_{ n \geq
 0} \) be the boundary points of the intervals \( (\Delta_n^{-})_{n \geq 0}\) and \(( \Delta_n^{+})_{n \geq 0}\) respectively such that
 \begin{equation}\label{eq:xn__}
\Delta_0^{-}=(x_{0}^{-}, 0), \qquad  \Delta_0^{+}=(0, x_{0}^{+}),
\qquad
\Delta_n^{-} = ( x_{n}^{-} , x_{n-1}^{-} ),
\qquad 
\Delta_n^{+} = ( x_{n-1}^{+} , x_n^{+}) 
\end{equation}
for $n\ge 1$.
\medskip

\noindent The following result gives the asymptotic rates of convergence of  \( ( x_n^{-} )\) and \( ( x_n^{+} )\) to -1 and 1 respectively.

\begin{prop}\label{Est_x_n}
We have the following asymptotics.
\begin{equation}\label{est_x_n}
    1+x_n^-\sim \ell_1^{-\frac{1}{\ell_1}}n^{-\frac{1}{\ell_1}}\mathcal{L}_{-1}^{\#}(n^{-\frac{1}{\ell_1}}), \quand 1-x_n^+\sim \ell_2^{-\frac{1}{\ell_2}}n^{-\frac{1}{\ell_2}}\mathcal{L}_{1}^{\#}(n^{-\frac{1}{\ell_2}}).
\end{equation}
 
\[|\Delta_n^-|\sim \ell_1^{-\frac{1}{\ell_1}-1}n^{-(\frac{1}{\ell_1}+1)}\mathcal{L}_{-1}^{\#}(n^{-\frac{1}{\ell_1}}) \quand |\Delta_n^+|\sim \ell_2^{-\frac{1}{\ell_2}-1}n^{-(\frac{1}{\ell_2}+1)}\mathcal{L}_{1}^{\#}(n^{-\frac{1}{\ell_2}}).\]
\end{prop}
\begin{proof}
Let \(z_n^-:  = 1+x_n^-\).
For the following let $\mathcal{R}_{-1,\ell_1}(x)=x^{\ell_1}\mathcal{L}_{-1}(x)$.
Then as  \begin{align*}
    z^-_n:  =& 1+g(x_{n+1}^-)
            = z^-_{n+1}(1+\mathcal{R}_{-1,\ell_1}(z^-_{n+1})),
\end{align*}
we have by Proposition 1 in \cite{holland2005slowly}
\[z^-_n\sim \big(\ell_1^{\frac{1}{\ell_1}}\bar{\mathcal{R}}^{-1}_{-1,\ell_1}(n)\big)^{-1},\]
where \(\bar{\mathcal{R}}_{-1,\ell_1}(n)\coloneqq\frac{1}{\mathcal{R}_{-1,\ell_1}(\frac{1}{n})}=n^{\ell_1}\bar{\mathcal{L}}_{-1}(n)\) and \(\bar{\mathcal{L}}_{-1}(n)=\left(\mathcal{L}_{-1}(\frac{1}{n})\right)^{-1}\) and $\bar{\mathcal{R}}^{-1}_{-1,\ell_1}$ is understood as the generalized inverse of $\bar{\mathcal{R}}_{-1,\ell_1}$ unique up to asymptotic equivalence. 
Together with Lemma \ref{1.5.15} we get \(\bar{\mathcal{R}}^{-1}_{-1,\ell_1}({n})=n^{\frac{1}{\ell_1}}\big(\bar{\mathcal{L}}_{-1}\big)^{\#}(n^{\frac{1}{\ell_1}})\).
Then the fact that $1/L^{\#}=(1/L)^{\#}$ for any slowly varying function $L$ implies
\[\bar{\mathcal{R}}^{-1}_{-1,\ell_1}({n})=n^{\frac{1}{\ell_1}}\Big(\mathcal{L}_{-1}^{\#}(n^{-\frac{1}{\ell_1}})\Big)^{-1}\]
and thus 
\[
z^-_n\sim \ell_1^{-\frac{1}{\ell_1}}n^{-\frac{1}{\ell_1}}\mathcal{L}_{-1}^{\#}(n^{-\frac{1}{\ell_1}}).
\]
\noindent By the mean value theorem, we have 
\[
|\Delta_n^-|=|x_{n-1}^{-}-x_{n}^{-}|=|z^-_{n-1}-z^-_n|\sim \frac{{\mathrm{d}}z^-_n}{{\mathrm{d}}n}
\]
which together with \eqref{est_x_n} gives the required estimates for $|\Delta_n^-|$. {Here, we note that $z_n^-$ is of course a function in $\mathbb{N}$. However, $\mathcal{L}_-$ is defined on the whole interval $[0,1]$ and thus $z_n^-$ could also be analytically extended to $\mathbb{R}$ making $\mathrm{d} z^-_n/\mathrm{d} n$ a meaningful expression.}

\noindent We obtain estimates for $1-x_n^{+}$ and $|\Delta_n^+|$ by an analogous argument.
\end{proof}
\medskip
\noindent Let 
\( ( y_n^{-} )_{n \geq 0} \text{ and }  ( y_{n}^+ )_{ n \geq
 0} \) be the boundary points of the intervals \( (\delta_n^{-})_{n \geq 0}\) and \(( \delta_n^{+})_{n \geq 0}\) respectively such that
 \begin{equation}\label{eq:yn__}
\delta_n^{-} = [ y_{n-1}^{-}, y_{n}^- ) \quand 
 \delta_n^{+} = ( y_{n}^{+}, y_{n-1}^{+} ]
\end{equation}
for $n\ge 1$.
\medskip
\begin{cor}\label{est_y_n__}
{We have the following asymptotics:}
\begin{equation}\label{Est_y_n__}
    y_n^-\sim - a_1^{-\frac{1}{k_1}}\ell_2^{-\frac{1}{\beta_2}}n^{-\frac1{\beta_2}}\big(\mathcal{L}_{1}^{\#}(n^{-\frac{1}{\ell_2}})\big)^{\frac{1}{k_1}} \quad \text{ and }\quad y_n^+\sim a_2^{-\frac{1}{k_2}}\ell_1^{-\frac{1}{\beta_1}}n^{-\frac1{\beta_1}}\big(\mathcal{L}_{-1}^{\#}(n^{-\frac{1}{\ell_1}})\big)^{\frac{1}{k_2}},
\end{equation} 
 
\[|\delta_n^-|\sim a_1^{-\frac{1}{k_1}}\beta_2^{-1}\ell_2^{-\frac{1}{\beta_2}}n^{-(1+\frac1{\beta_2})}\big(\mathcal{L}_{1}^{\#}(n^{-\frac{1}{\ell_2}})\big)^{\frac{1}{k_1}} \quad\text{ and }\quad |\delta_n^+|\sim a_2^{-\frac{1}{k_2}}\beta_1^{-1}\ell_1^{-\frac{1}{\beta_1}}n^{-(1+\frac1{\beta_1})}\big(\mathcal{L}_{-1}^{\#}(n^{-\frac{1}{\ell_1}})\big)^{\frac{1}{k_2}}.\]
\end{cor}
\begin{proof}
By our topological construction \(g(y_n^+)=x_{n-1}^-\) and the definition of $g$ in \eqref{eqn_1___} we have that
\(y_n^+=\big(\frac{1}{a_2}(1+x_{n-1}^-)\big)^{\frac{1}{k_2}}.\)

\noindent Together with Proposition \ref{Est_x_n} we get the estimate of $y_n^+$ in \eqref{Est_y_n__}.
By the mean value theorem we have \[|\delta_n^+|=|y_{n-1}^+-y_{n}^+|\sim \frac{{\mathrm{d}}y_{n}^+}{{\mathrm{d}}n}\]
\noindent which together with \eqref{Est_y_n__} and Lemma \ref{1.5.8} gives the required estimates for $|\delta_n^+|$. We obtain estimates for $y_n^{-}$ and $|\delta_n^-|$ by an analogous argument.
\end{proof}
\subsection{Expansion Estimates}\label{sec:exp__}
We start this subsection by providing some elementary estimates on $g'$ and $g''$. For \( \ell_{1}>0\), we have by Lemma \ref{1.5.8}
that
\begin{equation}\label{deq_4}
g'|_{U_{-1}}(x)\sim 1+(1+\ell_1)(1+x)^{\ell_1}\mathcal{L}_{-1}(1+x)
\quad\text{ and } \quad 
g''|_{U_{-1}}(x) \sim (1+\ell_1)\ell_1 (1+x)^{\ell_1-1}\mathcal{L}_{-1}(1+x),
\end{equation}
for $x$ tending to $-1$. We remark that although Lemma \ref{1.5.8} is stated for regularly varying functions in infinity, we can still use it for regular variation at zero, see Remark \ref{Rel_slow@0_and_infity}.

By the precise form of \eqref{eqn_1___}, we have \( g'(-1)= 1 \)  and thus the fixed point \( -1 \) is a \emph{neutral} fixed point.   Similarly, since \( \ell_{2}>0\) the fixed point \( 1 \) is a neutral fixed point as well.

Notice that changing the parameter values \(k_{1}, k_{2} \) gives rise to maps with quite different characteristics. For all  \( k_{1}>0 \) we have 
\begin{equation}\label{deq_5}
g'|_{U_{0-}}(x)= a_1k_1|x|^{k_1-1}
\quad\text{ and } \quad 
g''|_{U_{0-}}(x)= a_1k_1(k_1-1)|x|^{k_1-2}. 
\end{equation}
  So \( k_{1} \in (0,1) \) implies that \( |g'|_{U_{0-}}(x)|\to \infty \) as \( x \to 0 \), thus \( g'|_{U_{0-}} \) has a \emph{singularity} at 0 (one-sided), while \( k_{1} > 1 \)   implies that \( |g'|_{U_{0-}}(x)|\to 0\) as \( x \to 0 \), thus we say that \( g|_{U_{0-}} \) has a \emph{critical point} at 0 (one-sided). Analogous observations  hold for the different values of \( k_2\) and 
  Figure \ref{fig:fig} shows the graph of \( g \) for various combinations of the parameters $\ell_1,\ell_2,k_1,k_2,\mathcal{L}_{-1},\mathcal{L}$. 
  We also note that by \eqref{deq_5} Adler's condition first introduced in \cite{adler}, i.e.\ uniform boundedness of $|g''(x)|/g'(x)^2$, can not hold and one thus expects a different behaviour of these maps as the one studied e.g.\ in \cite{thaler_transformations_1983}.

For future reference we mention additional properties which follow from  \ref{itm:A1'}. 
Observe that if \( \ell_1 \in (0,1{]}\) we have  \(g''(x) \to { \infty}\) but if  \( \ell_1 >1 \) we have  \(g''(x) \to { 0}\),  as \( x\to -1\) and, as we shall see, this qualitative difference in the higher order derivative plays a crucial role in the ergodic properties of \( g \). 
Here, we notice that the behaviour for $\ell_1=1$ is determined by the slowly varying function $\mathcal{L}_{-1}$.
Similar observations hold for \( g|_{U_{1}}\). Also observe that for every \( x\in U_{-1}\)  we have 
 
    \begin{equation}\label{eq:2ndder1'}
{g''(x)}/{g'(x)}  \lesssim (1+x)^{\ell_1-1}\mathcal{L}_{-1}(1+x)
\end{equation}
and for every $x\in U_{0^+}$   
\begin{equation}\label{deq_55_}
{|g''(x)|}/{|g'(x)|} \lesssim  x^{-1}.
\end{equation}

The just established estimates will help us to prove the following proposition.
\begin{prop}\label{prop:expansion-new}
  For every \( g \in \mathfrak G \) the first return map \( G : \Delta_0^- \to \Delta_0^- \) given as in \eqref{eq:G__} is uniformly expanding.
\end{prop}

\noindent For proving this proposition we first define \( \phi : \Delta_0^+ \setminus \delta_1^+ \to \Delta_0^+ \) explicitly by 
\begin{equation}\label{eq:defphi}
 \phi \coloneqq (g|_{U_{0+}})^{-1} \circ g|_{U_{-1}} \circ g|_{U_{0+}}
\end{equation}
which is a bijection given by the commutative diagram in Figure \ref{fig:def-of-phi}.
\begin{figure}[h]
  \centering
  \begin{tikzcd}
      \delta_{n + 1}^+  \arrow[r, "\phi"] \arrow[d, "g"] & 
      \delta_{n}^+  \arrow[d, "g"] &
       \\
      \Delta_{n}^-  \arrow[r, "g"] &
      \Delta_{n-1}^-  & 
  \end{tikzcd}
  \caption{\label{fig:def-of-phi} \( \phi\) satisfies \( g\circ g = g \circ \phi \).}
\end{figure}

\begin{sublem}\label{lem:expansion}
For all \( x \in \delta_{n+1}^+\),  \( n \geq n^+\) (with $n_+$ as in \eqref{eq:n+-_}) we have 
  \[
    (g\circ g)'(x) > g'(\phi(x)).
  \]
\end{sublem}  
\begin{proof}
  We will prove the equivalent statement
  \begin{equation}\label{eq:phiexpansion}
   \frac{ g'(x) }{ g'(\phi(x))} g'(g(x)) > 1. 
  \end{equation}
  By the definition of $g$ in $U_{ 0+}$  in \eqref{eqn_1___} we have that
   \begin{equation}\label{eq:phix}
     \frac{g'(x)}{g'(\phi(x))} = \left(\frac{ x }{ \phi (x) }\right)^{ k_2 - 1}= \left(\frac{ \phi (x) }{ x }\right)^{1-k_2}.
   \end{equation}
If $k_2<1$, we already have - since $\phi(x)>x$ - that $g'(x)/g'(\phi(x))>1$. So we only need that $g'(g(x))>1$ which can be deduced from \eqref{eq: slow var cond}. 

Let us now assume that $k_2\geq 1$.      
The form of $g$ in $U_{0+}$  given in \ref{itm:A1'} together with \eqref{eq:defphi} gives
\begin{equation}\label{eq:def-phi-on-delta_n+1}
 \phi (x)=[x^{k_2}+a_2^{\ell_1}x^{k_2(1+\ell_1)}\mathcal{L}_{-1}(a_2x^{k_2})]^{\frac{1}{k_2}}\quad \text{ and thus }\quad \left( \frac{\phi (x)}{x} \right)^{k_2}  = 1 +  a_2^{\ell_1} x^{k_2\ell_1}\mathcal{L}_{-1}(a_2x^{k_2}).
\end{equation}
By \eqref{eq: slow var cond} we have 
$$ g'(g(x))-1\geq \epsilon \left(\left( \frac{\phi (x)}{x} \right)^{k_2}-1\right).$$
This and \eqref{eq:phix} imply 
   \begin{align*}
   \frac{ g'(x) }{ g'(\phi(x))} g'(g(x)) 
   &\geq  \left(\frac{ \phi (x) }{ x }\right)^{1-k_2}\left(\epsilon\left( \left( \frac{\phi (x)}{x} \right)^{k_2}-1\right)\right)\\
   &\geq  \left(\frac{ \phi (x) }{ x }\right)^{1-k_2}\left(\left( \frac{\phi (x)}{x} \right)^{k_2} - (1-\epsilon)\left( \left( \frac{\phi (x)}{x} \right)^{k_2}-1\right)\right)\\
   &=  \frac{ \phi (x) }{ x } - (1-\epsilon)\left(\left( \frac{\phi (x)}{x} \right)^{1-k_2}\left( \left( \frac{\phi (x)}{x} \right)^{k_2}-1\right)\right).
   \end{align*}
 Furthermore, we have 
 \begin{align*}
   \frac{ \phi (x) }{ x }&=1+ \left(\left(\frac{ \phi (x) }{ x }\right)^{k_2}\right)^{1/k_2}- 1^{1/k_2}\geq 1+ \left(\left(\frac{ \phi (x) }{ x }\right)^{k_2}\right)^{1/k_2-1}\left( \left( \frac{ \phi (x) }{ x }\right)^{k_2} -1\right)\\
   &= 1+ \left(\frac{ \phi (x) }{ x }\right)^{1- k_2}\left( \left( \frac{ \phi (x) }{ x }\right)^{k_2} -1\right)
 \end{align*}
 and thus 
   \begin{align*}
   \frac{ g'(x) }{ g'(\phi(x))} g'(g(x)) 
   &\geq 1+ \epsilon\left(\left( \frac{\phi (x)}{x} \right)^{1-k_2}\left( \left( \frac{\phi (x)}{x} \right)^{k_2}-1\right)\right)\geq 1.
   \end{align*}
  \end{proof}
\begin{proof}[Proof of Proposition \ref{prop:expansion-new}]
It suffices to show uniform expansion for \((\widetilde G^{+}){'}\) because it holds for \((\widetilde G^{-}){'}\) via identical procedure and \(G'=((\widetilde G^{-})'\circ\widetilde G^{+})\cdot  (\widetilde G^{+})'\).

Let \( x \in \delta_{n+1}^+ \),   \( n < n^+\). Then \((\widetilde G^{+}){'}\) is uniformly expanding by \ref{itm:A2'}. 

Suppose next \( n \geq n^+\). 
Then, by the definition of $\phi$ we have for any $1<j<n$ that \[g'(g^{j}(x))=g'(g^{j-1}(\phi(x))).\] 
This implies for any \( 1 \leq m \leq n \) that   
\begin{align*}
\begin{split}
    (g^{m+1}) ' (x) &= g'(x) g'(g(x)) \cdots g'(g^{m} (x))
    = \frac{g'(x)g'(g(x))}{ g'(\phi(x))}(g^m)'(\phi(x))> (g^m)'(\phi(x))\,,
  \end{split}      
  \end{align*}
where the last inequality follows from \eqref{eq:phiexpansion}. 
Thus, $(\tilde{G}^+)'(x) > (\tilde{G}^+)'(\phi(x))$.
Applying this to  $x\in \delta_{n^++1}^+$ and proceeding inductively for larger $n$ gives the desired result.
\end{proof}
\subsection{Distortion Estimates}\label{sec:dist'__}
\begin{prop}\label{prop:boundeddist_}
For all \( g \in \widehat{\mathfrak{G}} \) there exists a constant \( \mathfrak D>0\) such that for all \( 0\leq m < n \) and all \( x, y\in \delta^{\pm}_n\), 
\[
\log \dfrac{\left(g^{n-m}\right)'(g^m(x))}{\left(g^{n-m}\right)'(g^m(y))}
\leq  \mathfrak D |g^n(x) - g^n(y)|.
\]
\end{prop}

\begin{proof}
We assume for simplicity that \( x, y\in \delta^{+}_n\), the estimates for \( \delta^{-}_n \) are the same. 
With the same approach as in \cite[Lemma 3.1]{coaluzmub22} we first give a uniform bound for the left handside and then improve it to also include the factor $|g^n(x) - g^n(y)|$.
By the chain rule, we write
\[
\log\dfrac{\left(g^{n-m}\right)'(g^m(x))}{\left(g^{n-m}\right)'(g^m(y))}=\log\prod_{i=m}^{n-1}\dfrac{g'(g^i(x))}{g'(g^i(y))}
=\sum_{i=m}^{n-1}\log\dfrac{g'(g^i(x))}{g'(g^i(y))}.
\]
Without loss of generality we assume $x\leq y$.
Since \( (g^i(x), g^i(y))\subset \Delta^{-}_{n-i}\) (a smooth component of \(g\)) for \(0\leq i<n\), by the mean value theorem there is \(u_i\in\left(g^i(x), g^i(y)\right)\) such that
\[
\log \dfrac{g'(g^i(x))}{g'(g^i(y))} 
=  \log g'(g^i(x))-\log g'(g^i(y)) = \dfrac{g''(u_i)}{g'(u_i)} |g^i(x)-g^i(y)|.
\]
We substitute this into the above expression taking \( \mathfrak D_i:= {g''(u_i)}/{g'(u_i)}\) to get
\begin{equation}\label{eq:dist1_}
\log\dfrac{\left(g^{n-m}\right)'(g^m(x))}{\left(g^{n-m}\right)'(g^m(y))}
= \sum_{i=m}^{n-1} \mathfrak D_i |g^i(x)-g^i(y)| 
\leq \sum_{i=0}^{n-1} \mathfrak D_i |g^i(x)-g^i(y)|\leq  \mathfrak D_0 |\delta^{+}_n| + \sum_{i=1}^{n-1} \mathfrak D_i |\Delta^{-}_{n-i}|.
\end{equation}
By estimates in Proposition \ref{Est_x_n} and the bound \eqref{eq:2ndder1'} and finally applying the definition of the de Bruijn conjugate we have 
\begin{align*}
 \mathfrak D_i 
 &\lesssim  (1+x^-_{n-i})^{\ell_1-1}\mathcal{L}_{-1}(1+x^-_{n-i})\notag\\
 &\lesssim  (n-i)^{-1+\frac{1}{\ell_1}} \mathcal{L}_{-1}^\#((n-i)^{-\frac{1}{\ell_1}})^{\ell_1-1}\mathcal{L}_{-1}((n-i)^{-\frac{1}{\ell_1}}\mathcal{L}_{-1}^\#(i^{-\frac{1}{\ell_1}}))\notag\\
 &\sim (n-i)^{-1+\frac{1}{\ell_1}}(\mathcal{L}_{-1}^\#((n-i)^{-\frac{1}{\ell_1}}))^{\ell_1-2},
\end{align*}
for $n-i$ tending to $\infty$. 
Furthermore, applying Proposition \ref{Est_x_n} to $|\Delta_{n-i}^{-}|$ implies
\begin{align}
 \mathfrak D_i|\Delta_{n-i}^{-}|\lesssim (n-i)^{-2}
 (\mathcal{L}_{-1}^{\#}((n-i)^{-\frac{1}{\ell_1}}))^{\ell_1-1}.\label{eq:di1_}
\end{align}

Using Corollary \ref{est_y_n__} 
we have that $|\delta_n^+|\to0$ as $n\to\infty$ and thus
\begin{equation}\label{eq:di1__}
\sup_{n\in\mathbb{N}}\mathfrak D_0 |\delta_n^{+}|<\infty.
\end{equation}

\noindent Substituting \eqref{eq:di1__} and \eqref{eq:di1_}
into \eqref{eq:dist1_} we get
\begin{equation}\label{unif_bd}
  \sup_{m,n\in\mathbb{N}, m< n}\log\dfrac{\left(g^{n-m}\right)'(g^m(x))}{\left(g^{n-m}\right)'(g^m(y))}\lesssim \sup_{n\in\mathbb{N}}\left(\mathfrak D_0 |\delta^{+}_n| +\sum_{i=1}^{n-1}i^{-2}\mathcal{L}_{-1}^\#(i^{-\frac{1}{\ell_1}})^{\ell_1-1}\right)=:\widehat{ \mathfrak D }<\infty.  
\end{equation}

 Boundedness for the first summand is immediate. For the second summand there exists $C>0$ such that
\[\sum_{i=1}^{n-1}i^{-2} \mathcal{L}_{-1}^\#(i^{-\frac{1}{\ell_1}})^{\ell_1-1}
< \sum_{i=1}^{\infty}i^{-2} \mathcal{L}_{-1}^\#(i^{-\frac{1}{\ell_1}})^{\ell_1-1}<C\sum_{i=1}^{\infty}i^{-3/2}<\infty,\] since by Lemma \ref{lem: conv xlL(x)} the slowly varying function fulfills \(\mathcal{L}_{-1}^\#(i^{-\frac{1}{\ell_1}})^{\ell_1-1}\lesssim i^{1/2}\), for $i$ sufficiently large.

\noindent 
Next, we will improve the last estimate to include the factor $|g^n(x) - g^n(y)|$.
Together with the mean value theorem we have that the diffeomorphisms \( g^{n}: \delta_{n}^+ \to \Delta_0^-\) and \( g^{n-m}: \Delta_{n-m}^- \to \Delta_0^-\) all have uniformly bounded distortion; i.e.\ \[
|x-y|
\leq  \frac{e^{\widehat{\mathfrak D}}}{|\Delta_0^-|} 
|g^n(x)-g^n(y)|{|\delta_{n}^+|}
\quand 
|g^i(x)-g^i(y)|
\leq  \frac{e^{\widehat{\mathfrak D}}}{|\Delta_0^-|} 
|g^n(x)-g^n(y)|
{|\Delta_{n-m}^-|}
\] for all \( x,y \in \delta^+_n\) and \( 1\leq m < n \). Substituting back into \eqref{eq:dist1_}

\[\begin{aligned}
\log\frac{\left(g^{n-m}\right)'(g^m(x))}{\left(g^{n-m}\right)'(g^m(y))}
& \leq \frac{e^{\widehat{\mathfrak D}}}{|\Delta_0^-|} 
 \left[
 \mathfrak D_0  
 {|\delta_{n}^+|}
+ \sum_{i=1}^{n-1} \mathfrak D_i 
{|\Delta_{n-i}^-|}
\right] |g^n(x)-g^n(y)|
\\ & \leq
\dfrac{e^{\widehat{\mathfrak D}}\widehat{\mathfrak D}}{|\Delta_0^-|}  |g^n(x)-g^n(y)| 
= \mathfrak D |g^n(x)-g^n(y)|,
\end{aligned}
\]
where  \( \mathfrak D:= {e^{\widehat{\mathfrak D}}\widehat{\mathfrak D}}/{|\Delta_0^-|} \). Notice that in the first inequality the term in bracket{s} corresponds to \eqref{unif_bd} which is uniformly bounded.
\end{proof}

\noindent As an immediate consequence we get the following two corollaries. 
\begin{cor}\label{cor:lebesgue-measure-of-delta-ij_}
For all \( i,j \geq 1\) we have 
\begin{equation}\label{eq:dist___}
\tilde \mu( \delta_{i{,}j}^-) 
\lesssim i^{-(1+\frac1{\beta_2})}L_1(i)j^{-(1+\frac1{\beta_1})}L_{-1}(j).
\end{equation}
\end{cor}
\begin{proof}
Proposition \ref{prop:boundeddist_} implies that 
\[\frac{1}{\mathfrak D |\Delta_{0}^{+}|} |\delta_{i}^{-}| |\delta_{j}^{+}| \leq |\delta_{i,j}| 
\leq \frac{\mathfrak D}{ |\Delta_{0}^{+}|} |\delta_{i}^{-}| |\delta_{j}^{+}|.\] Applying estimates in Corollary \ref{est_y_n__}  gives the result as \( \tilde \mu \) is equivalent to Lebesgue on \( \Delta_0^-\cup \Delta_{0}^{+} \).
\end{proof}

\begin{cor}
  \label{cor:bounded-distortion}
There exist $C>0$ and $\theta\in (0,1)$ such that  
\[\log \bigg|\frac{G'(x)}{G'(y)}\bigg|\leq  C\theta^{s(x,y)}
\] for all $x, y\in\delta_{i,j}$ with  \( x \neq y\) where $s$ is the separation time defined in \eqref{sep-time}.
\end{cor}
\begin{proof}
Same as Corollary 3.11 in \cite{coaluzmub22}.
\end{proof}

\begin{proof}[Proof of Proposition \ref{thm:inducedmap'}]
    Let $G\coloneqq G^-$ (be the induced map which saturates $I$ constructed in Section \ref{sec:top__} which obviously coincides with the definition in \eqref{eq:G__}). As $G$ is full branch,
    \(G|_{\delta_{i,j}}: \delta_{i,j} \to \Delta_0^- \) is a \( C^{2}\) diffeomorphism and $G$ is uniformly expanding by Proposition \ref{prop:expansion-new}, we may conclude - using  Corollary \ref{cor:bounded-distortion} - that $G$ is a Gibbs-Markov map as in Definition \ref{def_GM}.
\end{proof}
\section{Tail Estimates and proof of Proposition \ref{tail_proposition_}}\label{sec:tail__est}

In this section we prove Proposition \ref{tail_proposition_}. We start first with the following lemma:
\begin{lem}\label{tail_lemma_}We have
\[\sum_{\substack{i<n, j<n\\i+j=n}}\tilde\mu(\delta_{i,j}^-)=o(\max\{|y_n^{+}|,|y_n^{-}|\} ).\]
\end{lem}

\begin{proof}
Remember the definition given in \eqref{eq: def L pm 1}
and note that both functions are slowly varying. 
By bounded distortion \eqref{eq:dist___} we have 
\begingroup
\allowdisplaybreaks
\begin{align}
   \sum_{\substack{i<n, j<n\\i+j=n}}\tilde\mu(\delta_{i,j}^-)
   &\lesssim  \sum_{i=1}^{n} i^{-(1+\frac1{\beta_2})}L_1(i)\sum_{j=n-i+1}^{n} j^{-(1+\frac1{\beta_1})}L_{-1}(j)\notag\\
   &\lesssim  \sum_{i=1}^{n} i^{-(1+\frac1{\beta_2})}L_1(i)\min\left\{ i (n-i)^{-(1+\frac1{\beta_1})}L_{-1}(n-i),(n-i)^{-\frac1{\beta_1}}L_{-1}(n-i)\right\}\notag\\
   &\lesssim  \sum_{i=1}^{\lfloor n/2 \rfloor} i^{-\frac1{\beta_2}}L_1(i)  (n-i)^{-(1+\frac1{\beta_1})}L_{-1}(n-i)\notag\\
   &\qquad +\sum_{i=\lfloor n/2 \rfloor+1}^n i^{-(1+\frac1{\beta_2})}L_1(i) (n-i)^{-\frac1{\beta_1}} L_{-1}(n-i)\notag\\
   &\lesssim \sum_{i=1}^{\lfloor n/2 \rfloor} i^{-\frac1{\beta_2}} L_1(i)  n^{-(1+\frac1{\beta_1})} L_{-1}(n)\notag\\
   &\qquad   + \sum_{i=\lfloor n/2 \rfloor+1}^n n^{-(1+\frac1{\beta_2})} L_1(n) (n-i)^{-\frac1{\beta_1}} L_{-1}(n-i)\notag\\
   &\lesssim \begin{cases}
       n^{-(1+\frac1{\beta_1})} L_{-1}(n) &\text{if }\sum_{k=1}^{\infty}\chi_k^1<\infty\\
       n^{-(1+\frac1{\beta_1})}\widetilde{L}_{-1}(n) &\text{if }\sum_{k=1}^{\infty}\chi_k^1=\infty\text{ and }\beta_2=1\\
       n^{-(\frac1{\beta_1}+\frac1{\beta_2})}L_{-1}(n)L_{1}(n) &\text{if }\beta_2>1
   \end{cases}\notag\\
   &\qquad+ \begin{cases}
       n^{-(1+\frac1{\beta_2})}L_1(n)&\text{if }\sum_{k=1}^{\infty}\chi_k^2<\infty\\
       n^{-(1+\frac1{\beta_2})}\widetilde{L}_{1}(n)&\text{if }\sum_{k=1}^{\infty}\chi_k^2=\infty\text{ and }\beta_1=1\\
       n^{-(\frac1{\beta_1}+\frac1{\beta_2})}L_{-1}(n)L_{1}(n) &\text{if }\beta_1>1,
   \end{cases}
    \label{small_o_1}
\end{align}
\endgroup
where $\widetilde{L}_{1}$ and $\widetilde{L}_{-1}$ are slowly varying functions which we do not define any further.
Here, we have made use of Lemma \ref{1.5.10} for the second inequality and of Lemma \ref{1.5.9b} in the second cases and of Lemma \ref{1.5.10} in the third cases in the brackets of the last inequality. 
 
By Corollary \ref{est_y_n__} we have
\begin{align*}
   \MoveEqLeft\frac{1}{\max\{|y_n^{+}|,|y_n^{-}|\}} \sum_{\substack{i<n, j<n\\i+j=n}}\tilde\mu(\delta_{i,j}^-)\notag\\
   &\lesssim \min\left\{ n^{\frac{1}{\beta_1}}L_{-1}(n), n^{\frac{1}{\beta_2}}L_1(n)\right\}  \sum_{\substack{i<n, j<n\\i+j=n}}\tilde\mu(\delta_{i,j}^-)=o(1).
\end{align*}
\end{proof}

\noindent Finally, we are able to prove Proposition \ref{tail_proposition_}.
\begin{proof}[Proof of Proposition \ref{tail_proposition_}]
Analogously to the approach in \cite{sandropaper} we can split $\tilde \mu(\tau>n)$ into the following four sums:
\begin{align}\label{tail_proof}
    \begin{split}
        \tilde \mu(\tau>n)=\sum_{i>n}\sum_{j=1}^{\infty}\tilde\mu(\delta_{i,j}^-)+\sum_{j>n}\sum_{i=1}^{\infty}\tilde\mu(\delta_{i,j}^-)-\sum_{i>n}\sum_{j>n}\tilde\mu(\delta_{i,j}^-)+\sum_{\substack{i<n,j<n\\i+j=n}}\tilde\mu(\delta_{i,j}^-).
    \end{split}
\end{align}
Since by the topology of our partition \(\delta_{i}^{-}= \bigcup_{j=1}^{\infty} \delta_{i,j}\) and  \(\bigcup_{i=n}^{\infty} \delta_{i}^{-}=[y_n^-,0)\), 
Lemma 4.6 in \cite{coaluzmub22} implies that the first sum in \eqref{tail_proof} is 
\[\sum_{i>n}\sum_{j=1}^{\infty}\tilde\mu(\delta_{i,j}^-)=\sum_{i>n}\tilde\mu(\delta_{i}^{-})=\tilde\mu((y_n^-,0))\sim \tilde h( 0^-) y_{n}^-.\]

\noindent Since the measure is invariant, we have by Lemma 4.3 and Lemma 4.6 in \cite{coaluzmub22} that the second sum fulfills
\[\sum_{j>n}\sum_{j=1}^{\infty}\tilde\mu(\delta_{i,j}^-)=\tilde\mu((0,y_n^+))\sim\tilde h( 0^+)y_{n}^+.\] 
\noindent Since the fourth sum decays faster than the sum of the first two by Lemma \ref{tail_lemma_}, we only have to show 
\begin{equation}\label{eq: to show}
 \sum_{i>n}\sum_{j>n}\tilde\mu(\delta_{i,j}^-)=o(\max\{|y_n^{+}|, |y_n^-|\}).   
\end{equation}
By bounded distortion \eqref{eq:dist___} we have
\begin{align*}
\MoveEqLeft\frac{1}{\max\{|y_n^{+}|, |y_n^-|\}}\sum_{i>n}\sum_{j>n}\tilde\mu(\delta_{i,j}^-)\notag\\
&\lesssim \frac{1}{\max\{|y_n^{+}|, |y_n^-|\}}\left(\sum_{i>n}i^{-(1+\frac1{\beta_2})}L_1(n)\right)\left(\sum_{j>n}j^{-(1+\frac1{\beta_1})}L_{-1}(n)\right)\\
&\lesssim \frac{1}{\max\{|y_n^{+}|, |y_n^-|\}}\left(n^{-\frac1{\beta_2}}L_1(n)\right)\left(n^{-\frac1{\beta_1}}L_{-1}(n)\right)\\
&\lesssim \min\left\{ n^{-\frac1{\beta_2}}L_1(n),\, n^{-\frac1{\beta_1}}L_{-1}(n) \right\}=o(1),
\end{align*}
establishing \eqref{eq: to show}.
We applied Lemma \ref{1.5.10} to get the second inequality and Corollary \ref{est_y_n__} to get the last inequality.
\end{proof}

\section{Proof of Theorem \ref{thm:main2''} and the examples}\label{sec: proof of thm D}
\subsection{Proof of Theorem \ref{thm:main2''}}
\begin{proof}[Proof of Theorem \ref{thm:main2''}]
Let \( g \in \mathfrak G\) and  \( G: \Delta_0^- \to \Delta_0^-  \) be its associated induced (Gibbs-Markov) map given in \eqref{eq:G__} with \(\mathscr P^- :=  \{ \delta_{m,n}^{-}\}_{m,n\geq 1}\), a $\mathrm{mod }\, 0$ partition of \(\Delta_0^-\) and $\tau:\Delta_0^-\to \mathbb{N}$, the return time function given in \eqref{eq: tau} which fulfills   \(\tau|_{\delta_{m,n}^{-} }= {m+n} \). 
The main idea is to use a slight adaption of \cite[Theorem 3 and Corollary 1]{holland2005slowly} as it was done in \cite[Theorem 1]{holland2005slowly}.
So, to obtain decay of correlations of $g$ with respect to the measure $\mu_g$ we associate a \emph{tower} to the induced map $G$ as follows.
We start by defining the \emph{tower} \[\Delta=\{(x,l): x\in\Delta_0^- \text{ and } l\in \mathbb{N}_0\cap [ 0,\tau(x))\},\]
the \emph{tower map} $T:\Delta\to \Delta$ of $G$
\[T(x,l):=\begin{cases}
    (x,l+1), \quad \text{if}\quad l<\tau(x)-1,\\
    (G(x),0), \quad \text{if}\quad l=\tau(x)-1,
\end{cases}\]
and for each $l\in\mathbb{N}_0$ the $l^{th}$ \emph{level of the tower} \[\Delta_{l}=\{(x,l): (x,l)\in\Delta\}.\]
Then \cite[Assumption 5]{holland2005slowly} is fulfilled by Corollary \ref{cor:bounded-distortion} and the fact that $\log(x)\leq |1-x|$ for any non-negative $x$. Assumption 6 follows then immediately by construction and Assumption 7 from the fact that we were choosing $g\in \mathfrak{G}$ and the fact that the density is equivalent to Lebesgue.  

We extend the measure $\hat\mu$ on \(\Delta_0^-\) to a measure  $\mathfrak{m}$ on \(\Delta\) where we take $\mathfrak{m}$ to be the product measure of $\hat{\mu}$ and the Lebesgue measure $\lambda$
and define a return time function $R:\Delta\to \mathbb{N}$ analogously to \(\tau\)  as 
 \[R (x,l)\coloneqq\min\{n\ge 0: T^n(x,l)=(G(x),0)\}=\tau(x)-l\] and $\mathfrak{m}\{R>n\}=\sum_{l\ge n} \hat{\mu}(\Delta_l)$. Notice that since \( g \in \mathfrak G\),
\[ \mathfrak{m}(\Delta)=\sum_{l\ge 0} \hat{\mu}(\Delta_l)=\sum_{l\ge 0}\hat\mu\{x\in\Delta_0^-:\tau(x)\geq l\}\lesssim \max\bigg\{\sum_{n=1}^{\infty} \chi^{1}_{n}, \sum_{n=1}^{\infty} \chi^{2}_{n}\bigg\}<\infty\,,\]
i.e.\  \(\mathfrak{m}\) is finite on \(\Delta\).
Together with Proposition \ref{tail_proposition_} we have that 
\[\mathfrak{m}\{R>n\}=\sum_{l\ge n} \mathfrak{m}(\Delta_l)=\sum_{l\ge n}\hat\mu\{\tau>l\}\sim C\sum_{l\ge n} l^{-\frac1{\beta}}L(l)\]
with $C>0$ and $L$ defined as in \eqref{gen_SVF}.

If $\beta=1$, we have that $\mathfrak{m}\{R>n\}$ is slowly varying and by construction it is also monotonically decreasing, so we can apply \cite[Theorem 3]{holland2005slowly}. If $\beta<1$, then $\mathfrak{m}\{R>n\}$ is regularly varying as in  \cite[Corollary 1]{holland2005slowly}.

However, those theorems only give us a decay of correlation result with respect to $T$ and $\mathfrak{m}$. 
So, in the next steps we deduce from the tails of the return time function we have estimated on the tower to the correlation on the map itself, similarly as on \cite[p.\ 148]{holland2005slowly}.

Let \(\varphi_1,\varphi_2\in \mathcal{H}\). Define 
\[\bar{\varphi}_1=\varphi_1\circ\pi,\qquad \bar{\varphi}_2=\varphi_2\circ\pi:\Delta\to \mathbb{R},\] 
where \(\pi:\Delta\to [-1,1]\) is the tower projection defined as \(\pi(x,l)=g^l(x)\). Using a change of variables, we have 
\begin{align*}
\int_{\Delta_0^-}(\varphi_1\circ g^n\varphi_2)\, \mathrm{d}\mu_g
&=\int_{\Delta_0^-}(\bar{\varphi}_1\circ\pi^{-1}\circ g^n)(\bar{\varphi}_2\circ\pi^{-1})\,\mathrm{d}\mu_g
=\int_{\Delta}(\bar{\varphi}_1\circ\pi^{-1}\circ g^n\circ\pi)(\bar{\varphi}_2\circ\pi^{-1}\circ \pi)\,\mathrm{d}\mathfrak{m}\\
&=\int_{\Delta}(\bar{\varphi}_1\circ T^n)\bar{\varphi}_2 \,\mathrm{d}\mathfrak{m}.
\end{align*}
Hence, we obtain 
\begin{align*}
  \left|\int_{\Delta_0^-}(\varphi_1\circ g^n\varphi_2)\, \mathrm{d}\mu_g  - \int_{\Delta_0^-}\varphi_1\, \mathrm{d}\mu_g \int_{\Delta_0^-}\varphi_2\, \mathrm{d}\mu_g\right|
  &=\left| \int_{\Delta}(\bar{\varphi}_1\circ T^n)\bar{\varphi}_2 \,\mathrm{d}\mathfrak{m}- \int_{\Delta}\bar{\varphi}_1 \,\mathrm{d}\mathfrak{m}\int_{\Delta}\bar{\varphi}_2 \,\mathrm{d}\mathfrak{m} \right|
\end{align*}
and thus by using the estimates given by Theorem 3 and Corollary 1 of \cite{holland2005slowly} for the right hand side we obtain the statement of the theorem.
\end{proof}

\subsection{Proof of Examples \ref{ex: Example1} to \ref{ex: Example3}}
The proofs all follow the same scheme which we first give here. The first step is to show that $\mu_g$ is finite. This is the case if there exist $n_0,\bar{n}_0$ such that $\max\{\sum_{n=\bar{n}_0}^{\infty} \chi_n^1,\sum_{n=n_0'}^{\infty}\chi_n^2\}<\infty$ where $n_0'=\max\{n_0,\bar{n}_0'\}$. 
 Since we assume $\beta=1$, this is equivalent to $\sum_{n=n_0'}^{\infty} L(n)/n<\infty$, where 
 $L$, combining \eqref{eq: def L pm 1} and \eqref{gen_SVF} can be written as 
 \begin{align}
  L(n)= \begin{cases}
   \mathcal{L}^{\#}_{-1}(n^{-\frac{1}{\ell_1}})^{\frac{1}{k_2}}&\text{if }\beta_2>\beta_1\\
   \mathcal{L}^{\#}_{1}(n^{-\frac{1}{\ell_2}})^{\frac{1}{k_1}}&\text{if }\beta_1>\beta_2\\
   \max\left\{ \mathcal{L}^{\#}_{-1}(n^{-\frac{1}{\ell_1}})^{\frac{1}{k_2}}\, ,\, \mathcal{L}^{\#}_{1}(n^{-\frac{1}{\ell_2}})^{\frac{1}{k_1}}\right\}&\text{if }\beta_1=\beta_2.
  \end{cases}\label{eq: L calculated}
 \end{align}
Then, in order to estimate the decay of correlations we may use Theorem \ref{thm:main2''} and we have to calculate $\sum_{l\ge n}l^{-1}L(l)$ with $L$ as above. 
Hence, the main task to prove the examples is to calculate $L$.

\begin{proof}[Proof of Example \ref{ex: Example1}]
 In order to determine the de Bruijn conjugate of $\mathcal{L}^{\#}_{-1}$, $\mathcal{L}^{\#}_{1}$ we use B\'ek\'essy's criterion, see e.g.\ \cite[Appendix 5.2]{bingham1989regular}, implying that $\mathcal{L}^{\#}_{\pm 1}(n)\sim (\mathcal{L}_{\pm 1})^{- 1}$. Thus, in case $\beta_1>\beta_2$  we have using \eqref{eq: def kappa} and \eqref{eq: L calculated} that
 \begin{align}
  L(n)=\mathcal{L}^{\#}_{-1}(n^{-\frac{1}{\ell_1}})^{\frac{1}{k_2}}
  &\sim \mathcal{L}_{-1}(n^{-\frac{1}{\ell_1}})^{-\frac{1}{k_2}}
  = \left((\log^r (n^{\frac{1}{\ell_1}}))^{k_2+\alpha} \prod_{j=1}^{r-1}(\log^j (n^{\frac{1}{\ell_1}}))^{k_2}\right)^{-\frac{1}{k_2}}\notag\\
  &\sim \frac{1}{\ell_1}(\log^r (n))^{-1-\frac{\alpha}{k_2}} \prod_{j=1}^{r-1}(\log^j (n))^{-1}.\label{eq: deBruijnConj}
 \end{align}
 An analogous calculation gives the case $\beta_2>\beta_1$ and in case $\beta_1=\beta_2$ we have 
 \begin{align*}
 L(n)
 &=\max\left\{ \frac{1}{\ell_1}(\log^r (n))^{-\frac{\alpha+k_2}{\kappa}} \prod_{j=1}^{r-1}(\log^j (n))^{-\frac{k_2}{\kappa}}\, ,\, \frac{1}{\ell_2}(\log^r (n))^{-\frac{\alpha+k_1}{\kappa}} \prod_{j=1}^{r-1}(\log^j (n))^{-\frac{k_1}{\kappa}}\right\}\\
 &\lesssim (\log^r (n))^{-1-\frac{\alpha}{\kappa}} \prod_{j=1}^{r-1}(\log^j (n))^{-1}.
\end{align*}
Hence, $\sum_{n=n_0'}^{\infty} L(n)/n<\infty$ if and only if $\alpha>0$ and calculating $\sum_{l\geq n}L(l)/l$ gives the desired result for the decay of correlation. 
\end{proof}

\begin{proof}[Proof of Example \ref{ex: Example2}]
 We first calculate the de Bruijn conjugate of $\mathcal{L}^{\#}_{-1}$ for which we use
 (1) of Lemma \ref{lem: lagrange inversion}. 
 If we set 
 $$h(x)=\kappa \Big((\log^r (x))^{\alpha}+(1-\alpha)\log^{r+1} (x)+\sum_{j=1}^{r-1}\log^{j+1}(x)\Big),$$ then 
 $\mathcal{L}_{-1}(x)=\exp(h(x))$ and since $h'(x)\to 0$ and $h(x)h'(x)\to 0$ we obtain $\mathcal{L}^{\#}_{-1}(n)\sim (\mathcal{L}_{-1})^{-1}$. Thus, if $\beta_1>\beta_2$, \eqref{eq: def kappa} implies
 \begin{align}
  L(n)&=\mathcal{L}^{\#}_{-1}(n^{-\frac{1}{\ell_1}})^{\frac{1}{k_2}}
  \sim \mathcal{L}_{-1}(n^{-\frac{1}{\ell_1}})^{-\frac{1}{k_2}}
  = \left(\exp\left((\log^r (n^{\frac{1}{\ell_1}}))^{\alpha}\right)\, (\log^r (n^{\frac{1}{\ell_1}}))^{1-\alpha}\prod_{j=1}^{r-1}\log^j(n^{\frac{1}{\ell_1}})\right)^{-\frac{\kappa}{k_2}}\notag\\
  &\sim \left(\exp\left((\log^r (n))^{\alpha}\right)\, (\log^r (n))^{1-\alpha}\prod_{j=1}^{r-1}\log^j(n)\right)^{-1}.\label{eq: L calculated 1}
 \end{align}
 The case $\beta_2>\beta_1$ follows analogously and for the case $\beta_1=\beta_2$ we obtain 
  \begin{align*}
  L(n)&=\max\bigg\{ \Big(\exp\left((\log^r (n^{\frac{1}{\ell_1}}))^{\alpha}\right)\, (\log^r (n^{\frac{1}{\ell_1}}))^{1-\alpha}\prod_{j=1}^{r-1}\log^j(n^{\frac{1}{\ell_1}})\Big)^{-\frac{\kappa}{k_2}}
  \,,\\
  &\qquad\qquad\,\,
  \Big(\exp\left((\log^r (n^{\frac{1}{\ell_2}}))^{\alpha}\right)\, (\log^r (n^{\frac{1}{\ell_2}}))^{1-\alpha}\prod_{j=1}^{r-1}\log^j(n^{\frac{1}{\ell_1}})\Big)^{-\frac{\kappa}{k_1}}\bigg\},
 \end{align*}
giving the same expression as in \eqref{eq: L calculated 1}. 
Moreover, we may easily calculate that $(\exp(-(\log^r (n))^{\alpha}))'=L(n)$ implying both that the measure $\mu_g$ is finite and we have the decay of correlations as claimed in  \eqref{eq: dec exm2}.
\end{proof}

\begin{proof}[Proof of Example \ref{ex: Example3}]
In order to obtain the de Bruijn conjugate we will use Lemma \ref{lem: lagrange inversion} where in case of $\alpha<1/2$ we may set 
$$h(x)=\kappa\left(\lambda^{\alpha} x^{\alpha} +(1-\alpha)\log (x)\right)
\quad\text{ obtaining }\quad 
h'(x)= \kappa\left(\lambda^{\alpha} x^{\alpha-1}+(1-\alpha)x^{-1}\right)$$ 
and thus
$h'(x)\to 0$ and $h(x)h'(x)\to 0$ as long as $\alpha<1/2$. 
Hence, Lemma \ref{lem: lagrange inversion} implies $\mathcal{L}^{\#}_{-1}(n)\sim (\mathcal{L}_{-1})^{-1}$. 
Thus, if $\beta_1>\beta_2$, using \eqref{eq: def kappa} and \eqref{eq: def lambda} implies
 \begin{align*}
  L(n)&=\mathcal{L}^{\#}_{-1}(n^{-\frac{1}{\ell_1}})^{\frac{1}{k_2}}
  \sim \mathcal{L}_{-1}(n^{-\frac{1}{\ell_1}})^{-\frac{1}{k_2}}
  =\exp\left(\kappa\left(\alpha\lambda^{\alpha} (\log (n^{\frac{1}{\ell_1}}))^{\alpha}+(1-\alpha)\log\log (n^{\frac{1}{\ell_1}})\right)\right)^{-\frac{1}{k_2}}\\
  &=\exp\left(-\frac{\kappa}{k_2}\left(\left(\frac{\lambda}{\ell_1}\right)^{\alpha} (\log (n))^{\alpha}\right)\right)\left(\frac{\log (n)}{\ell_1}\right)^{\alpha-1}
  =\ell_1^{1-\alpha}\exp\left(- (\log (n))^{\alpha}\right)\log (n)^{\alpha-1} 
 \end{align*}
 and the case $\beta_2>\beta_1$ follows analogously. 
 If $\beta_1=\beta_2$, we obtain 
  \begin{align*}
  L(n)& =\max\bigg\{\exp\left(-\frac{\kappa}{k_2}\left(\left(\frac{\lambda}{\ell_1}\right)^{\alpha} (\log (n))^{\alpha}\right)\right)\left(\frac{\log (n)}{\ell_1}\right)^{\alpha-1}
  \, ,\\
  &\qquad\qquad\,\, \exp\left(-\frac{\kappa}{k_1}\left(\left(\frac{\lambda}{\ell_2}\right)^{\alpha} (\log (n))^{\alpha}\right)\right)\left(\frac{\log (n)}{\ell_2}\right)^{\alpha-1}\bigg\}
  \asymp \exp\left(- (\log (n))^{\alpha}\right)\log (n)^{\alpha-1}.
 \end{align*}
 Noticing that $(\exp(-(\log n)^{\alpha}))'=L(n)$, gives the statement of the lemma for $\alpha\in (0,1/2)$. 
\end{proof}

\begin{rem}
 It would also be possible to give similar calculations for $\alpha>1/2$. In the examples we were giving we were in the easy situation that the de Bruijn conjugate could simply be given by the reciprocal. However, as can easily be seen, if we consider $\alpha>1/2$ in the last example, the condition $h(x)h'(x)\to 0$ would be violated. However, for $\alpha\in[1/2, 2/3)$ it would be possible to apply (2) of Lemma \ref{lem: lagrange inversion} and \cite{bingham1989regular} gives explanations how this lemma can be generalizes even if (2) is violated. However, in that case the terms become increasingly complicated while the new insight from such examples remains limited.
\end{rem}

\appendix
\section{Results for regularly and slowly varying functions}
In this section we list some relevant results on regularly and slowly varying functions we have used. Although the results stated here are for functions $\mathcal{L}$ slowly varying at \(\infty\), by setting $\widetilde{\mathcal{L}}(x)=\mathcal{L}(1/x)$ we have that $\widetilde{\mathcal{L}}$ is a slowly varying function in zero, see Remark \ref{Rel_slow@0_and_infity}.

\begin{lem}[Theorem 1.4.1 (Characterization of regularly varying functions)\cite{bingham1989regular}]\label{xtic_RV}
Every function $\mathcal{R}$ regularly varying at $\infty$ with index $\gamma$ can be written as $\mathcal{R}(x)=x^{\gamma}\mathcal{L}(x)$ where $\mathcal{L}$ is slowly varying.
\end{lem}

\begin{lem}[Theorem 1.3.1 \label{rep_SV}(Representation of slowly varying functions)\cite{bingham1989regular}]
Let $\mathcal{L}:\mathbb{R}^+\to \mathbb{R}^+$ be slowly varying at $\infty$. Then there exist $C_0>0$ and measurable functions $\epsilon,\zeta:[C_0,\infty)\to \mathbb{R}^+$ 
$\zeta(x)\to c\in \mathbb{R}^+,  \epsilon(x)\to 0$ as $x\to \infty$ such that
\[\mathcal{L}(x)=\zeta(x)\exp\left\{\int_{C_0}^x\frac{\epsilon(t)}{t}\,\mathrm{d}t\right\} \text{ for all } x\geq C_0.\]
\end{lem}

\begin{lem}[Proposition 1.5.1 \label{lem: conv xlL(x)}\cite{bingham1989regular}]
 Let $\mathcal{L}$ be a slowly varying function at $\infty$ and let $\ell\neq 0$, then 
 \[
 x^{\ell}\mathcal{L}(x)\to \begin{cases}
   \infty& \text{if }\ell>0\\
   0&\text{if }\ell<0,
 \end{cases}
 \]
 as $x\to\infty$.
\end{lem}

\begin{lem}[Theorem 1.5.13 \cite{bingham1989regular}]\label{1.5.13}
If \(\mathcal{L}\) is slowly varying at $\infty$, there exists a slowly varying function \(\mathcal{L}^\#\), unique up to asymptotic equivalence such that \[\mathcal{L}(x)\mathcal{L}^\#(x\mathcal{L}(x))\to 1, \quad \mathcal{L}^\#(x)\mathcal{L}(x\mathcal{L}^\#(x))\to 1,\]
as $x\to\infty$
and \(\mathcal{L}^{\#\#}=\mathcal{L}\).
\end{lem}

\begin{lem}[Theorem 1.5.15 \cite{bingham1989regular}]\label{1.5.15} Let $a,b>0$ and $g(x)\sim x^{a} \mathcal{L}^a(x^b)$ where $\mathcal{L}$ is slowly varying at $\infty$. Suppose $f$ is an asymptotic inverse of $g$ (i.e.\  $g(f(x))\sim f(g(x))\sim x$). Then \[f(x)\sim x^{\frac{1}{ab}} \mathcal{L}^{\#\frac{1}{b}}(x^\frac1a).\]
\end{lem}

\begin{lem}[Proposition 1.5.8 \cite{bingham1989regular}]\label{1.5.8}
If $\mathcal{L}$ is slowly varying at $\infty$, $M$ is so large that $\mathcal{L}(x)$ is locally bounded in $[M,\infty)$, and $\alpha>-1$, then
\[\int_M^{x}t^{\alpha}\mathcal{L}(t)\,\mathrm{d}t\sim x^{\alpha+1}\mathcal{L}(x)/(\alpha+1) \quad\text{ as }\quad x\to 0.\]
\end{lem}

\begin{lem}[Proposition 1.5.9b \cite{bingham1989regular}]\label{1.5.9b}
If $\mathcal{L}$ is slowly varying and there exists $M>0$ such that $\int_M^{\infty}\mathcal{L}(t)/t\,\mathrm{d}t<\infty$, then $\overline{\mathcal{L}}(x):= \int_x^{\infty} \mathcal{L}(t)/t\,\mathrm{d}t$ is slowly varying and $\lim_{x\to\infty}\overline{\mathcal{L}}(x)/\mathcal{L}(x)=\infty$.
\end{lem}

\begin{lem}[Proposition 1.5.10 \cite{bingham1989regular}]\label{1.5.10}
If $\mathcal{L}$ is slowly varying and $\alpha<-1$ then \(\int^{\infty}t^{\alpha}\mathcal{L}(t)\,\mathrm{d}t\) converges and 
\[\dfrac{x^{\alpha+1}\mathcal{L}(x)}{\int_x^{\infty}t^{\alpha}\mathcal{L}(t)\,\mathrm{d}t}\to -\alpha-1 \quad\text{ as }\quad x\to \infty.\]
\end{lem}

\begin{lem}[{Lagrange inversion formula, application of Theorem A.5.2 \cite{bingham1989regular}}]\label{lem: lagrange inversion}
 Let us assume that $\mathcal{L}$ is slowly varying in $\infty$ and can be written as $\mathcal{L}(x)=\exp(h(\log (x)))$ with $h$ a holomorphic function. 
 \begin{enumerate}
  \item If $h'(x)\to 0$ and $h(x)h'(x)\to 0$ as $x\to\infty$, then 
  $$\mathcal{L}^{\#}(x)\sim \left(\mathcal{L}(x)\right)^{-1}.$$
  \item If $h'(x)\to 0$ and $h(x)(h'(x))^2\to 0$ and $h^2(x)h''(x)\to0$ as $x\to\infty$, then 
  $$\mathcal{L}^{\#}(x)\sim \left(\mathcal{L}(x)\right)^{-1+2h'(\log (x))}.$$
 \end{enumerate}
\end{lem}
The statement is stated slightly differently in \cite{bingham1989regular} but the version ofb Lemma \ref{lem: lagrange inversion} follows by a simple calculation.

\begin{rem}\label{Rel_slow@0_and_infity}
Assuming that $\mathcal{L}$ is a slowly varying function in $\infty$, by setting $\widetilde{\mathcal{L}}(x)=\mathcal{L}(1/x)$ we have that $\widetilde{\mathcal{L}}$ is a slowly varying function in zero. For this function analogous results can be stated. 
In particular, the same characterization statement holds and we have for the de Bruijn conjugate that  $\widetilde{\mathcal{L}}^{\#}(x)=\mathcal{L}^{\#}(1/x)$ is a possible choice and thus analogous statements for Lemma \ref{1.5.13} and Lemma \ref{1.5.15} can be stated for convergence $x\to 0$.

Moreover, the statement of Lemma \ref{lem: conv xlL(x)} changes for $\ell\neq 0$ into 
\[
x^{\ell}\widetilde{\mathcal{L}}(x)\to \begin{cases} 
0&\text{if }\ell>0\\
\infty&\text{if }\ell<0,
\end{cases}
\]
as $x\to 0$. The analogs of 
Lemmas \ref{1.5.8} and \ref{1.5.10} can be deduced by noting that if $\ell \neq 0$ we have $(x^{\ell}\widetilde{\mathcal{L}}(x))'=((1/x)^{-\ell}\mathcal{L}(1/x))'\sim -\ell (1/x)^{-\ell-1}\mathcal{L}(1/x) (-1) x^{-2}= \ell x^{\ell-1} \widetilde{\mathcal{L}}(x)$. 
Here, to conclude from the function $x^{\alpha}\widetilde{\mathcal{L}}(x)$ to its primitive we have to assume local boundedness if $\alpha<-1$.
\end{rem}

\medskip

\medskip

\noindent
\textbf{Acknowledgements}
\\
The authors would like to thank Stefano Luzzatto, Douglas Coates and Mark Holland
for their comments and interesting discussions.
\\
\\

\noindent
The authors declare that there are no competing interests related to the results in this paper.
\printbibliography
\end{document}